\documentclass[reqno,12pt,letterpaper]{amsart}
\usepackage{amsmath,amssymb,amsthm,graphicx,mathrsfs,url}
\usepackage[usenames,dvipsnames]{color}
\usepackage[colorlinks=true,linkcolor=Red,citecolor=Green]{hyperref}
\usepackage{amsxtra}
\usepackage{wasysym} 
\usepackage{graphicx}
\usepackage{subcaption}

\usepackage{graphicx,color}

\usepackage{array}
\newcolumntype{P}[1]{>{\centering\arraybackslash}m{#1}}

\def\wrtext#1{\relax\ifmmode{\leavevmode\hbox{#1}}\else{#1}\fi}

\def\abs#1{\left|#1\right|}

\def\begeq{\begin{equation}}
\def\endeq{\end{equation}}

\def\?[#1]{\textbf{[#1]}\marginpar{\Large{\textbf{??}}}}

\let\epsilon=\varepsilon 

\def\norm#1{||\,#1\,||}

\newtheorem{dref}{Definition}[section]

\newtheorem{theo}[dref]{Theorem}

\newtheorem{prop}[dref]{Proposition}

\numberwithin{equation}{section}

\usepackage{scalerel}

\newcommand\reallywidehat[1]{\arraycolsep=0pt\relax%
\begin{array}{c}
\stretchto{
  \scaleto{
    \scalerel*[\widthof{\ensuremath{#1}}]{\kern-.5pt\bigwedge\kern-.5pt}
    {\rule[-\textheight/2]{1ex}{\textheight}} 
  }{\textheight} %
}{0.5ex}\\           
#1\\                 
\rule{-1ex}{0ex}
\end{array}
}

\begin{document}

\title[Magnetic translations and pseudodifferential operators]{Weighted symbol spaces, magnetic translations, and pseudodifferential operators}

\author{Michael Hitrik}
\address{Department of Mathematics, University of California, Los Angeles, CA 90095, USA.}
\email{hitrik@math.ucla.edu}

\author{Reid Johnson}
\address{Department of Mathematics, University of California, Los Angeles, CA 90095, USA.}
\email{reid@math.ucla.edu}

\begin{abstract}
We study pseudodifferential operators associated to microlocally defined normed symbol spaces of limited regularity, introduced by J. Sj\"ostrand~\cite{Sj08}. Boun\-ded\-ness of such operators on modulation spaces is obtained under suitable conditions, and continuity properties of the Weyl composition product are established.
\end{abstract}

\maketitle

\setcounter{tocdepth}{1}
\tableofcontents

\section{Introduction and statement of results}
\label{sec_intro}
\medskip
\noindent
In 1994-95, in a series of two remarkable, now classic, papers~\cite{Sj94},~\cite{Sj95}, J. Sj\"ostrand introduced a Banach space of symbols of limited regularity on the phase space $E = T^*\mathbb R^n$, and established that the corresponding space of pseudodifferential operators is stable under composition, is contained in the space of $L^2$--bounded operators, and enjoys the Wiener property. The space of symbols in question, nowadays often called the Sj\"ostrand class~\cite{Gr06}, is defined as follows: let $\Gamma$ be a lattice in $E$ and let $\chi \in \mathscr S(E)$ be such that
\[
\sum_{\gamma \in \Gamma} \tau_{\gamma} \chi = 1, \quad \tau_{\gamma} \chi(x) = \chi(x - \gamma).
\]
We say that $a\in \mathscr S'(E)$ is in the Sj\"ostrand class if
\[
\sup_{\gamma \in \Gamma} \abs{\mathcal F(\chi_{\gamma} a)(x^*)} \in L^1(E^*). 
\]
Here $\mathcal F$ stands for the Fourier transformation on $E$. See also~\cite{Bou97}. It subsequently turned out that the Sj\"ostrand class  can be viewed as a special case of modulation spaces~\cite{Fe83}, \cite{BeOk_book}, and as such, it plays a significant role in time frequency ana\-ly\-sis~\cite{Gr_book}. The paper~\cite{GrHe99} established that pseudodifferential operators with symbols in the Sj\"ostrand class are bounded on the full scale of modulation spaces, and there is by now a rich and substantial body of work studying pseudodifferential and Fourier integral operators with symbols of limited regularity and their relations to broad classes of modulation spaces, see \cite{GrHe99}, \cite{To01}, \cite{La01}, \cite{To04}, \cite{To07}, \cite{HoToWa07}, \cite{GrRz08}, \cite{CoNi}, \cite{GrTo11}, \cite{Co13}. As remarked in~\cite{Sj95}, some of the original motivation for the algebras of pseudodifferential operators introduced in~\cite{Sj94} came from their potential applications to non-linear, including inverse, problems, as well as to problems in high dimensions. More recently, the Wiener algebras of~\cite{Sj94},~\cite{Sj95} were applied to the study of decay rates for the damped wave equation with low re\-gu\-larity damping, see~\cite{AnLe},~\cite{Kl}.

\medskip
\noindent
In 2008, in the work~\cite{Sj08}, Sj\"ostrand returned to the study of pseudodifferential operators with symbols of limited regularity, generalizing substantially the original definition of the symbol classes in~\cite{Sj94} and establishing results on the $L^2$ continuity and the composition for the associated pseudodifferential classes. The purpose of this note is to pursue the study of the classes of operators introduced in~\cite{Sj08}, obtaining continuity results on modulation spaces and making the results of~\cite{Sj08} on the composition slightly more precise. We shall now proceed to introduce the assumptions and state the main results. When doing so, we shall always consider the Weyl quantization
\[
a^w: \mathscr S(\mathbb R^n) \rightarrow \mathscr S'(\mathbb R^n),
\]
of a symbol $a\in \mathscr S'(T^*\mathbb R^n)$, given by
\[
a^w u(x) = \frac{1}{(2\pi)^n} \int\!\!\!\int e^{i(x-y)\cdot \theta} a\left(\frac{x+y}{2},\theta\right) u(y)\, dy\, d\theta.
\]

\bigskip
\noindent
Let $E = \mathbb R^d$, and let $0 < m$ be an order function on $E \times E^*$, where $E^*$ is the dual space, in the sense that $m$ is Lebesgue measurable and we have for some $C_0>0$, $N_0 > 0$,
\begeq
\label{eq1.0}
m(X) \leq C_0 \langle{X-Y\rangle}^{N_0} m(Y),\quad X,Y \in E \times E^*.
\endeq
Here we write $\langle{X -Y \rangle} = (1 + \abs{X -Y}^2)^{1/2}$. Let $\Gamma$ be a lattice in $E \times E^*$, so that
\[
\Gamma = \bigoplus_{j=1}^{2d} \mathbb Z e_j,
\]
where $e_1, \ldots ,e_{2d}$ is a basis for $E \times E^*$, and let $0\leq \chi \in \mathscr S(E \times E^*)$ be such that
\begeq
\label{eq1.1}
\sum_{\gamma \in \Gamma} \tau_{\gamma}\chi = 1, \quad \tau_{\gamma}\chi (X) = \chi(X - \gamma).
\endeq
When $a\in \mathscr S'(E)$, following~\cite[Definition 2.1]{Sj08} we say that $a\in \widetilde{S}(m)$ if the function
\begeq
\label{eq1.2}
\Gamma \ni \gamma \mapsto \frac{1}{m(\gamma)} \norm{\chi_{\gamma}^w a}_{L^2(E)}
\endeq
belongs to $\ell^{\infty}(\Gamma)$. Here $\chi_{\gamma} = \tau_{\gamma} \chi$ and $\chi_\gamma^w$ stands for the Weyl quantization of $\chi_{\gamma}$. It is shown in~\cite[Proposition 2.2]{Sj08} that when equipped with the norm given by the $\ell^{\infty}(\Gamma)$--norm of the function in (\ref{eq1.2}), the space $\widetilde{S}(m)$ becomes a Banach space, and changing $\Gamma$, $\chi$ gives rise to the same space with an equivalent norm.

\medskip
\noindent
{\it Remark}. It follows from~\cite[Proposition 5.1]{Sj08},~\cite[Proposition 6.1]{Horm91} that each $a\in \mathscr S'(E)$ satisfies $a\in \widetilde{S}(m)$, for some order function $m$ on $E \times E^*$. The space $\widetilde{S}(1)$ agrees with the modulation space $M^{\infty}(E)$, see (\ref{app30}) below, and in particular, $L^p(E) \subseteq \widetilde{S}(1)$, $\mathcal F(L^p(E)) \subseteq \widetilde{S}(1)$, for $1\leq p \leq \infty$. Here $\mathcal F u(\xi) = \widehat{u}(\xi) = \int e^{-ix\cdot \xi} u(x)\, dx$ is the standard Fourier transformation on $E$.

\medskip
\noindent
In what follows we shall take $E = T^*\mathbb R^n$, equipped with its standard symplectic structure, and let
\begeq
\label{eq1.3}
J=\begin{pmatrix}0 &1\\ -1 &0\end{pmatrix},\quad J^t = -J,\quad J^2 = -1.
\endeq
It will sometimes be convenient to view $J$ as a linear map: $E^* \ni \ell \mapsto J \ell = H_{\ell}\in E$, where $H_{\ell}$ is the Hamilton vector field of the linear form $\ell$. Let us introduce the linear bijection
\begeq
\label{eq1.3.1}
q: E \times E \ni (x,y) \mapsto \left(\frac{x+y}{2}, J^{-1}(y-x)\right) \in E \times E^*.
\endeq
The following is the first main result of this note, where $M^p(\mathbb R^n)$ stands for the standard (unweighted) modulation space, see~\cite[Chapter 11]{Gr_book} and Appendix \ref{modFBI}.

\begin{theo}
\label{theo_bounded}
Let $E = T^*\mathbb R^n$ and let $a\in \widetilde{S}(m)$, where $m$ is an order function on $E \times E^*$ such that the integral operator
\begeq
\label{eq1.4}
\mathcal Mf(x) = \int_E m(q(x,y)) f(y)\, dy
\endeq
is bounded on $L^p(E)$, for some $1\leq p \leq \infty$. Then the operator
\begeq
\label{eq1.5}
a^w: M^p(\mathbb R^n) \rightarrow M^p(\mathbb R^n)
\endeq
is bounded, and we have for some $C>0$ independent of $p$,
\begeq
\label{eq1.5.1}
\norm{a^w}_{\mathcal L(M^p(\mathbb R^n), M^p(\mathbb R^n))} \leq C \norm{a}_{\widetilde{S}(m)}\, \norm{\mathcal M}_{\mathcal L(L^p(E), L^p(E))}.
\endeq
In particular, if $\mathcal M$ is a Schur class operator,
\[
\sup_x\, \int_E m(q(x,y))\, dy < \infty, \quad \sup_y\, \int_E m(q(x,y))\, dx < \infty,
\]
then the operator $a^w$ is bounded,
\[
a^w: M^p(\mathbb R^n) \rightarrow M^p(\mathbb R^n), \quad \forall\, p\in [1,\infty],
\]
with
\[
\norm{a}_{\mathcal L(M^p(\mathbb R^n), M^p(\mathbb R^m))} \leq C \norm{a}_{\widetilde{S}(m)}\, {\rm max}\left(\sup_x \int_E m\left(q(x,y)\right)\, dy, \sup_y \int_E m\left(q(x,y)\right)\, dx\right).
\]
\end{theo}

\medskip
\noindent
{\it Remark}. In the case $p=2$, when $M^2(\mathbb R^n) = L^2(\mathbb R^n)$, Theorem \ref{theo_bounded} is due to Sj\"ostrand \cite[Theorem 3.1]{Sj08} and Bony \cite[Th\'eor\`eme 2.8]{Bo94}.

\medskip
\noindent
{\it Remark}. Since the Schwartz space $\mathscr S(\mathbb R^n)$ is not dense in $M^{\infty}(\mathbb R^n)$, the precise boundedness statement (\ref{eq1.5}) for $p = \infty$ is as follows: there exists a constant $C>0$ such that
\[
\norm{a^w u}_{M^{\infty}(\mathbb R^n)} \leq C \norm{a}_{\widetilde{S}(m)}\, \norm{\mathcal M}_{\mathcal L(L^{\infty}(E), L^{\infty}(E))} \norm{u}_{M^{\infty}(\mathbb R^n)}, \quad u\in \mathscr S(\mathbb R^n).
\]

\bigskip
\noindent
Regarding the composition of pseudodifferential operators with symbols in $\widetilde{S}(m)$--spaces, we have the following result, which is only a minor sharpening of~\cite[Theorem 4.2]{Sj08}, the only possibly new point here being the uniqueness and continuity properties of the Weyl product extended to $\widetilde{S}(m)$--spaces. See also~\cite[Remark 4.3]{Sj08}, where the possibility of such a sharpening is suggested. When stating the result, it will be convenient to use the following notion of sequential continuity. Let $E = \mathbb R^d$ and let $m_j$, $1\leq j \leq 3$, be order functions on $E \times E^*$. We say that a bilinear map
\[
B: \widetilde{S}(m_1) \times \widetilde{S}(m_2) \rightarrow \widetilde{S}(m_3)
\]
is weakly sequentially continuous if given bounded sequences $u_j \in \widetilde{S}(m_1)$, $v_j \in \widetilde{S}(m_2)$, such that $u_j \rightarrow u \in \widetilde{S}(m_1)$, $v_j \rightarrow v \in \widetilde{S}(m_2)$ in $\mathscr S'(E)$, it is true that $B(u_j,v_j)$ is bounded in $\widetilde{S}(m_3)$ and $B(u_j,v_j) \rightarrow B(u,v)$ in $\mathscr S'(E)$.

\begin{theo}
\label{theo_compos}
Let $E = T^*\mathbb R^n$ and let $m_1$, $m_2$ be order functions on $E\times E^*$. Let us set, recalling the linear bijection $q$ in {\rm (\ref{eq1.3.1})},
\begeq
\label{eq1.6}
m_3(q(x,y)) = \int_E m_1(q(x,z))\, m_2(q(z,y))\, dz.
\endeq
Assume that the integral in {\rm (\ref{eq1.6})} converges for at least one value of $(x,y) \in E \times E$. Then it converges for all values and defines an order function $m_3$ on $E\times E^*$. Furthermore, the Weyl composition map
\[
\mathscr S(E) \times \mathscr S(E) \ni (a_1,a_2) \mapsto a_1 \# a_2 \in \mathscr S(E),
\]
defined by $a_1^w\circ a_2^w = (a_1 \# a_2)^w$, has a unique bilinear weakly sequentially continuous extension
\[
\widetilde{S}(m_1) \times \widetilde{S}(m_2)\ni (a_1, a_2) \rightarrow a_1 \# a_2 \in \widetilde{S}(m_3),
\]
which is also norm continuous,
\[
\norm{a_1 \# a_2}_{\widetilde{S}(m_3)} \leq C\,\norm{a_1}_{\widetilde{S}(m_1)} \norm{a_2}_{\widetilde{S}(m_2)}, \quad C > 0.
\]
For $a_j \in \widetilde{S}(m_j)$, $j=1,2$, the composition $a_1^w \circ a_2^w: \mathscr S(\mathbb R^n) \rightarrow \mathscr S'(\mathbb R^n)$ is  well defined and the Weyl symbol is given by $a_1 \# a_2$.
\end{theo}

\medskip
\noindent
The proofs of Theorems \ref{theo_bounded} and \ref{theo_compos} follow the arguments of~\cite{Sj08} closely and in particular, an important role throughout is played by the techniques of metaplectic FBI-Bargmann transforms \cite{Sj96}, \cite{HiSj15}, \cite[Chapter 13]{Zw12}. We also make use of the rank one decomposition of pseudodifferential operators on the Bargmann transform side, established in~\cite{HiLaSjZe}, representing a pseudodifferential operator in the complex domain as an integral of a family of rank one operators, expressed in terms of unitary magnetic translations acting on coherent states. The work~\cite{HiLaSjZe} was concerned with continuity properties of semiclassical Gevrey pseudodifferential operators in the complex domain, acting on exponentially weighted spaces of entire functions, but the scope of the rank one decomposition is in fact quite general and we take advantage of it here.

\medskip
\noindent
The plan of this note is as follows. In Section \ref{sec_prelim}, we prove a natural density result for the space $\widetilde{S}(m)$ and recall the rank one decomposition for Weyl quantization on the FBI-Bargmann transform side, obtained in~\cite{HiLaSjZe}. Theorem \ref{theo_bounded} is then established in Section \ref{sect_bounded}, and Section \ref{sect_compos} is devoted to the proof of Theorem \ref{theo_compos}. In Appendix \ref{modFBI}, we recall the definition of modulation spaces by means of metaplectic FBI-Bargmann transforms and show their independence of the choice of such a transform.

\medskip
\noindent
{\bf Acknowledgments}. We are very grateful to Johannes Sj\"ostrand for stimulating discussions and for encouraging us to pursue this study.

\section{Symbol spaces and pseudodifferential operators}
\label{sec_prelim}
\setcounter{equation}{0}

\medskip
\noindent
Let $E = \mathbb R^d$. We shall first establish the following natural density result for the symbol space $\widetilde{S}(m)$, suggested in~\cite[Remark 4.3]{Sj08}.

\begin{prop}
\label{density_prop}
Let $u\in \widetilde{S}(m)$. There exists a sequence $u_\nu \in \mathscr S(E)$, $\nu = 1,2,\ldots $ converging to $u$ in $\mathscr S'(E)$, which is bounded in the Banach space $\widetilde{S}(m)$.
\end{prop}
\begin{proof}
We shall work with translation invariant partitions of unity of compact support. Let $J \subseteq E$ be a lattice and let $0\leq \chi_0 \in C_0^\infty(E)$ be such that
\begeq
\label{eq2.1}
\sum_{j\in J} \chi_j = 1, \quad \chi_j(x) = \chi_0(x-j).
\endeq
Given $u \in \widetilde{S}(m)$, we claim that we can find a sequence $u_\nu \in \mathscr S(E)$, $\nu = 1,2,\ldots $ converging to $u$ in $\mathscr S'(E)$ such that,
\begin{equation}
\label{eq2.2}
|\widehat{\chi_j u_\nu}(x^*)| \leq C\,m(j,x^*), \quad j\in J, \,\,x^* \in E^*,\,\,\,\nu =1,2,\ldots
\end{equation}
Here our constants $C>0$ may change from line to line and may depend on $u$, but never on $\nu$. The fact that $u_{\nu}$ is then bounded in $\widetilde{S}(m)$ follows from the proof of~\cite[Proposition 2.4]{Sj08}.

\medskip
\noindent
When verifying the claim, we proceed as in~\cite[Section 1]{Sj94}. Indeed, let $\psi \in \mathscr S(E)$ be such that $\psi(0) = 1$,
let $0 \leq \varphi \in C_0^\infty(E)$ satisfy $\int\varphi = 1$, and let us define, with $*$ denoting convolution,
\begeq
\label{eq2.2.1}
u_\nu = \psi\left(\frac{x}{\nu}\right)\left(u * \varphi_{\frac{1}{\nu}}\right) \in \mathscr S(E), \quad \nu = 1,2,\ldots
\endeq
Here $\varphi_{\frac{1}{\nu}}(x) = \nu^d \varphi(\nu x)$. We have $u_\nu \rightarrow u$ in $\mathscr S'(E)$, and to prove (\ref{eq2.2}) we first observe that we have, in view of (\ref{eq2.1}),
\[
\chi_j (u * \varphi_{\frac{1}{\nu}}) = \chi_j\left(\sum_{|k-j|\leq C}(\chi_k u) * \varphi_{\frac{1}{\nu}}\right),
\]
for some constant $C > 0$. Therefore,
\begeq
\label{eq2.3}
\mathcal{F}(\chi_j (u * \varphi_{\frac{1}{\nu}})) = (2\pi)^{-d} \widehat{\chi}_j * \left(\widehat{\varphi}_{\frac{1}{\nu}} \sum_{|k-j|\leq C} \widehat{\chi_k u}\right).
\endeq
When estimating (\ref{eq2.3}), we notice that $|\widehat{\chi}_j|=|\widehat{\chi}_0|$,
$$
\|\widehat{\varphi}_{\frac{1}{\nu}}\|_{L^\infty(E)} \leq \|\varphi_{\frac{1}{\nu}}\|_{L^1(E)} = \|\varphi\|_{L^1(E)} = 1,
$$
and also, using~\cite[Proposition 2.4]{Sj08} and the fact that $m$ is an order function, we get
\[
|\widehat{\chi_k u}(x^*)| \leq Cm(k, x^*) \leq C C_0 \langle k-j\rangle^{N_0} m(j, x^*) \\
\implies \abs{\sum_{|k-j|\leq C} \widehat{\chi_k u}(x^*)} \leq Cm(j, x^*).
\]
It follows that
\[
\abs{\mathcal{F}(\chi_j(u * \varphi_{\frac{1}{\nu}}))} \leq C \abs{\widehat{\chi}_0} * m(j, \cdot),
\]
and thus, recalling (\ref{eq2.2.1}), we get
\begeq
\label{eq2.4}
\abs{\widehat{\chi_j u_\nu}} \leq C \abs{\nu^d \widehat{\psi}(\nu \cdot)} * \abs{\widehat{\chi}_0} * m(j, \cdot).
\endeq
Here we observe, using that $\widehat{\chi}_0$ is Schwartz and that $m$ is an order function,
\begin{multline}
\label{eq2.5}
(\abs{\widehat{\chi}_0} * m(j,\cdot))(x^*) = \int_{E^*} \abs{\widehat{\chi}_0(y^*)} m(j, x^* - y^*)\, dy^* \\
\leq C\int_{E^*} \langle y^*\rangle^{-N_0-d-1} (m(j, x^*) \langle y^*\rangle^{N_0})\, dy^* = Cm(j, x^*),
\end{multline}
and thus we get, using (\ref{eq2.4}), (\ref{eq2.5}), and the rapid decay of $\widehat{\psi}$,
\begin{multline*}
\abs{\widehat{\chi_j u_\nu}(x^*)} \leq C \left(\abs{\nu^d \widehat{\psi}(\nu \cdot)} * m(j,\cdot)\right)(x^*) \leq
C \int_{E^*} |\nu^d \widehat{\psi}(\nu y^*)| m(j, x^* - y^*)\, dy^*
\\ \leq C\int_{E^*} (\nu^d \langle\nu y^*\rangle^{-N_0-d-1}) (m(j, x^*) \langle y^*\rangle^{N_0})\, dy^*
\\ \leq C\int_{E^*} (\nu^d\langle\nu y^*\rangle^{-N_0-d-1}) (m(j,x^*) \langle \nu y^*\rangle^{N_0})\, dy^* = Cm(j, x^*).
\end{multline*}
The proof is complete.
\end{proof}

\medskip
\noindent
{\it Remark}. Let $u_{\nu}$ be a bounded sequence in $\widetilde{S}(m)$, such that $u_{\nu} \rightarrow u$ in $\mathscr S'(E)$, for some $u \in \mathscr S'(E)$. It follows from~\cite[Proposition 5.1]{Sj08} that $u \in \widetilde{S}(m)$, with $\norm{u}_{\widetilde{S}(m)} \leq
\mathcal O(1)\, \liminf \, \norm{u_{\nu}}_{\widetilde{S}(m)}$.

\bigskip
\noindent
Next, given $a\in \mathscr S'(T^*\mathbb R^n)$, we shall recall the representation of the Weyl quantization $a^w$ of $a$ as a superposition of rank one operators on the FBI-Bargmann transform side, established in~\cite[Section 4]{HiLaSjZe}. When doing so, let $\varphi$ be a holomorphic quadratic form on $\mathbb C^n_x \times \mathbb C^n_y$ such that
\begeq
\label{eq2.6}
{\rm det}\, \varphi''_{xy} \neq 0, \quad {\rm Im}\, \varphi''_{yy}>0.
\endeq
Associated to $\varphi$ is the complex linear canonical transformation
\begeq
\label{eq2.7}
\kappa: T^* \mathbb C^{n} \ni (y,-\varphi'_y(x,y)) \mapsto (x,\varphi'_x(x,y)) \in T^*\mathbb C^{n},
\endeq
and the metaplectic FBI-Bargmann transform
\begeq
\label{eq2.8}
\mathcal T u(x) = C \int_{\mathbb R^n} e^{i\varphi(x,y)}\, u(y)\, dy,
\endeq
which is unitary,
\begeq
\label{eq2.9}
\mathcal T: L^2(\mathbb R^n) \rightarrow H_{\Phi}(\mathbb C^n),
\endeq
for a suitable choice of $C>0$ in (\ref{eq2.8}), see \cite[Section 1]{Sj96}, \cite[Section 12.2]{Sj02}, \cite[Section 1.3]{HiSj15}. Here
\begeq
\label{eq2.9.1}
H_{\Phi}(\mathbb C^n) = {\rm Hol}(\mathbb C^n) \cap L^2(\mathbb C^n, e^{-2\Phi}L(dx)),
\endeq
with $L(dx)$ being the Lebesgue measure on $\mathbb C^n$, and $\Phi$ is a strictly plurisubharmonic quadratic form on $\mathbb C^n$, given by
\begeq
\label{eq2.10}
\Phi(x) = \sup_{y \in \mathbb R^n} \left(-{\rm Im}\, \,\varphi(x,y)\right).
\endeq
We have furthermore, in view of~\cite[Proposition 1.3.2]{HiSj15},
\begeq
\label{eq2.11}
\kappa\left(T^*\mathbb R^{n}\right) = \Lambda_{\Phi} := \left\{\left(x,\frac{2}{i}\frac{\partial \Phi}{\partial x}(x)\right); x\in \mathbb C^n\right \} \subseteq \mathbb C^{2n},
\endeq
and it follows that the restriction of the complex symplectic (2,0)--form
\begeq
\label{eq2.12}
\sigma = \sum_{j=1}^n d\xi_j \wedge dx_j
\endeq
on $\mathbb C^{2n} = \mathbb C^n_{x} \times \mathbb C^n_{\xi}$ to the $2n$--dimensional real linear subspace $\Lambda_{\Phi}$ is real and non-degenerate. The restriction of a complex linear form $\ell(x,\xi) = \ell'_x \cdot x + \ell'_{\xi} \cdot \xi$ on $\mathbb C^{2n}$ to $\Lambda_{\Phi}$ is therefore real precisely when the Hamilton vector field $H_{\ell} = \ell'_{\xi}\cdot \partial_x -\ell'_x\cdot \partial_{\xi}$ of $\ell$ satisfies
\begeq
\label{eq2.12.0.1}
H_{\ell} \in \Lambda_{\Phi} \iff -\ell'_x = \frac{2}{i}\frac{\partial \Phi}{\partial x}(\ell'_{\xi}).
\endeq
We may write therefore
\begeq
\label{eq2.12.1}
\ell(x,\xi) = \ell'_x\cdot x + \ell'_{\xi}\cdot \xi = \sigma((x,\xi),H_{\ell}) = -\frac{2}{i}\frac{\partial \Phi}{\partial x}(x^*)\cdot x + x^* \cdot \xi,\quad (x,\xi) \in \mathbb C^{2n},
\endeq
for some unique $\mathbb C^n \ni x^* = \ell'_{\xi}$. Associated to the complex linear form $\ell$ in (\ref{eq2.12.1}) and the complex canonical transformation
\[
\exp(H_{\ell})(\rho) = \rho + H_{\ell},\quad \rho \in \mathbb C^{2n},
\]
is the magnetic translation operator $e^{-i\ell(x,D_x)} = e^{-i\sigma((x,D_x),H_{\ell})}$, unitary on $H_{\Phi}(\mathbb C^n)$, given by
\begeq
\label{eq2.12.2}
e^{-i\sigma((x,D_x),H_{\ell})} = e^{-\frac{i}{2} \ell'_x\cdot x} \circ \tau_{\ell'_{\xi}} \circ e^{-\frac{i}{2} \ell'_x\cdot x},
\endeq
see~\cite[Proposition 1.4]{Sj96},~\cite[Section 2]{HiLaSjZe}. Here $\tau_z$ is the operator of translation by $z\in \mathbb C^n$, $(\tau_z u)(x) = u(x-z)$.

\bigskip
\noindent
Let $e_0(x) = \pi^{-n/4} e^{-x^2/2}$, $x\in \mathbb R^n$, be the $L^2$--normalized real Gaussian, and let $f_0$ be the Weyl symbol of the orthogonal projection $L^2(\mathbb R^n) \ni u \mapsto (u,e_0)_{L^2(\mathbb R^n)}\, e_0$ onto $\mathbb C e_0$. We have
\begeq
\label{eq2.13}
f_0(x,\xi) = \int e^{-iy\cdot \xi} e_0\left(x + \frac{y}{2}\right) e_0\left(x - \frac{y}{2}\right)\, dy = 2^n e^{-(x^2 + \xi^2)}, \quad (x,\xi) \in T^* \mathbb R^n,
\endeq
and let us introduce, passing to the FBI-Bargmann transform side by means of (\ref{eq2.7}), (\ref{eq2.11}),
\begeq
\label{eq2.14}
\chi_0 = f_0 \circ \kappa^{-1} \in \mathscr S(\Lambda_{\Phi}), \quad \chi_T(X) = (\tau_T \chi_0)(X) = \chi_0(X-T), \quad X,T \in \Lambda_{\Phi}.
\endeq
Define also
\[
V_0 = \mathcal T e_0 \in H_{\Phi}(\mathbb C^n),
\]
and introduce rank one operators on $H_{\Phi}(\mathbb C^n)$,
\begeq
\label{eq2.14.1}
\Pi_{Y,T} U = (U, e^{i\sigma((x,D_x),T)}V_0)_{H_{\Phi}(\mathbb C^n)}\, e^{i\sigma((x,D_x),Y)}V_0,\quad Y,T \in \Lambda_{\Phi}.
\endeq

\begin{theo}
\label{rank_one}
Let $a\in \mathscr S'(T^*\mathbb R^n)$ and define $b = a\circ \kappa^{-1} \in \mathscr S'(\Lambda_{\Phi})$. We have for all $u,v\in \mathscr S(\mathbb R^n)$,
\begin{multline}
\label{eq2.15}
(a^w u,v)_{L^2(\mathbb R^n)}:= \langle{a^w u, \overline{v}\rangle}_{\mathscr S'(\mathbb R^n), \mathscr S(\mathbb R^n)} \\
= \frac{1}{2^{n} (2\pi)^{2n}}\int\!\!\!\int_{(\Lambda_{\Phi})^2} e^{\frac{i\sigma(Y,T)}{2}}
\mathcal F_{\sigma}\left(\chi_{-\frac{Y+T}{2}} b\right)\left(\frac{Y-T}{2}\right)(\Pi_{Y,T}\mathcal Tu, \mathcal Tv)_{H_{\Phi}(\mathbb C^n)}\, dY\, dT.
\end{multline}
Here $\mathcal F_{\sigma}$ is the symplectic Fourier transformation on the real symplectic vector space $(\Lambda_{\Phi}, \sigma|_{\Lambda_{\Phi}})$,
\begeq
\label{eq2.16}
\mathcal F_{\sigma} u(X) = \frac{1}{\pi^n} \int_{\Lambda_{\Phi}} e^{2i\sigma(X,Y)} u(Y)\, dY, \quad X \in \Lambda_{\Phi},
\endeq
with $dY$ being the symplectic volume form on $\Lambda_{\Phi}$. The integral in {\rm (\ref{eq2.15})} converges absolutely.
\end{theo}
\begin{proof}
The rank one decomposition (\ref{eq2.15}) has been established in~\cite[Section 4]{HiLaSjZe} for $a$ of Schwartz class, and to show that (\ref{eq2.15}) holds also for $a\in \mathscr S'(T^* \mathbb R^n)$, we shall proceed by an argument of approximation. When doing so, it will be convenient to observe first that we have for every $N \in \mathbb N$ and for each fixed $u$, $v\in \mathscr S(\mathbb R^n)$,
\begeq
\label{eq2.17}
(\Pi_{Y,T}\mathcal Tu, \mathcal Tv)_{H_{\Phi}(\mathbb C^n)} = \mathcal O_N(1) \langle{Y\rangle}^{-N}\,\langle{T\rangle}^{-N}, \quad Y,T \in \Lambda_{\Phi}.
\endeq
Indeed, an application of the method of quadratic stationary phase~\cite[Lemma 13.2]{Zw12} gives that
\begeq
\label{eq2.18}
V_0(x) = (\mathcal Te_0)(x) = C e^{ig(x)}, \quad 0 \neq C \in \mathbb C,
\endeq
where $g$ is a holomorphic quadratic form on $\mathbb C^n$ given by the critical value
\[
g(x) = {\rm vc}_y \left(\varphi(x,y) + iy^2/2\right).
\]
It follows that the complex Lagrangian plane
\[
\Lambda_g = \{(x,g'(x)); x \in \mathbb C^n\} \subseteq T^*\mathbb C^{n}
\]
satisfies $\Lambda_g = \kappa(\{(y,iy); y\in \mathbb C^n\})$. Here the complex Lagrangian plane $\eta = iy$, $y\in \mathbb C^n$, is strictly positive with respect to $T^*\mathbb R^{n}$~\cite[Proposition 1.2.5]{HiSj15}, and therefore $\Lambda_g$ is strictly positive with respect to $\Lambda_{\Phi}$, see~\cite[Definition 1.2.4]{HiSj15}. An application of~\cite[Proposition 1.2.8]{HiSj15} allows us to conclude that
\begeq
\label{eq2.19}
\Phi(x) + {\rm Im}\, g(x) \asymp \abs{x}^2, \quad x\in \mathbb C^n.
\endeq

\medskip
\noindent
Next, a straightforward computation performed in~\cite[Section 2]{HiLaSjZe}, which uses (\ref{eq2.12.0.1}) and (\ref{eq2.12.2}) only, shows that we have for $Y = (y,(2/i)\partial_y \Phi(y)) \in \Lambda_{\Phi}$,
\begeq
\label{eq2.20}
\abs{\left(e^{i\sigma((x,D_x),Y)} V_0\right)(x)} = e^{\Phi(x) - \Phi(x+y)} \abs{V_0(x+y)}, \quad x\in \mathbb C^n.
\endeq
We get therefore, combining (\ref{eq2.18}), (\ref{eq2.19}), and (\ref{eq2.20}), for some $C>0$,
\begeq
\label{eq2.20.1}
\abs{\left(e^{i\sigma((x,D_x),Y)} V_0\right)(x)}e^{-\Phi(x)} \leq C\, e^{-\abs{x+y}^2/C}, \quad x\in \mathbb C^n,
\endeq
and it follows from (\ref{eq2.20.1}) that for all $N \in \mathbb N$,
\begin{multline}
\label{eq2.21}
\abs{(e^{i\sigma((x,D_x),Y)} V_0, \mathcal Tv)_{H_{\Phi}(\mathbb C^n)}} \leq C \int e^{-\abs{x+y}^2/C}\, \abs{\mathcal Tv(x)}e^{-\Phi(x)}\, L(dx) \\
\leq \mathcal O_N(1) \int \langle{x\rangle}^{-N} e^{-\abs{x+y}^2/C}\, L(dx) \leq  \mathcal O_N(1)\langle{y\rangle}^{-N}.
\end{multline}
Here we have also used that $\mathcal Tv(x) = \mathcal O_N(1) \langle{x\rangle}^{-N} e^{\Phi(x)}$, $\forall N$, see~\cite[Proposition 6.1]{Horm91}. The rapid decay estimate (\ref{eq2.17}) follows from (\ref{eq2.21}), in view of (\ref{eq2.14.1}).

\medskip
\noindent
It is now easy to complete the proof, by an argument of approximation. Let $a\in \mathscr S'(T^*\mathbb R^n)$ and let $a_j \in \mathscr S(T^* \mathbb R^n)$ be such that $a_j \rightarrow a$ in $\mathscr S'(T^*\mathbb R^n)$. It follows that for all $u,v \in \mathscr S(\mathbb R^n)$, we have, writing $b_j = a_j \circ \kappa^{-1} \in \mathscr S(\Lambda_{\Phi})$,
\begin{multline}
\label{eq2.22}
(a^wu,v)_{L^2(\mathbb R^n)} = \lim_{j\rightarrow \infty} (a_j^wu,v)_{L^2(\mathbb R^n)} = \lim_{j\rightarrow \infty} (b_j^w\, \mathcal Tu,\mathcal Tv)_{H_{\Phi}(\mathbb C^n)} \\
= \lim_{j\rightarrow \infty} \frac{1}{2^{n} (2\pi)^{2n}} \int\!\!\!\int_{(\Lambda_{\Phi})^2} e^{\frac{i\sigma(Y,T)}{2}}
\mathcal F_{\sigma}\left(\chi_{-\frac{Y+T}{2}} b_j\right)\left(\frac{Y-T}{2}\right)(\Pi_{Y,T}\mathcal Tu, \mathcal Tv)_{H_{\Phi}(\mathbb C^n)}\, dY\, dT.
\end{multline}
Here we have
\[
\mathcal F_{\sigma}\left(\chi_{-\frac{Y+T}{2}} b_j\right)\left(\frac{Y-T}{2}\right) \rightarrow \mathcal F_{\sigma}\left(\chi_{-\frac{Y+T}{2}} b\right)\left(\frac{Y-T}{2}\right), \quad j \rightarrow \infty
\]
pointwise, and the Banach-Steinhaus theorem gives a uniform bound for some fixed $M\geq 0$,
\[
\abs{\mathcal F_{\sigma}\left(\chi_{-\frac{Y+T}{2}} b_j\right)\left(\frac{Y-T}{2}\right)} \leq \mathcal O(1) \langle{Y\rangle}^M \langle{T\rangle}^M, \quad j=1,2,\ldots,\quad Y,T \in \Lambda_{\Phi}.
\]
Recalling (\ref{eq2.17}) we obtain therefore (\ref{eq2.15}) from (\ref{eq2.22}) by dominated convergence.
\end{proof}

\medskip
\noindent
{\it Remark}. As also mentioned in~\cite[Section 4]{HiLaSjZe}, the rank one decomposition (\ref{eq2.15}) can be viewed as an FBI-Bargmann transform side analogue of the corresponding decomposition in the real domain, established in~\cite{GrHe99}, and the observation here is that working in the complex domain, on the FBI–Bargmann transform side, seems to make the estimates and computations particularly natural and direct. See also~\cite[Section 5]{Sj95},~\cite[Section 5]{Sj08}.

\section{Boundedness on modulation spaces}
\label{sect_bounded}
\medskip
\noindent
The purpose of this section is to prove Theorem \ref{theo_bounded}. When doing so, we shall first recall the FBI transform characterization of the symbol class $\widetilde{S}(m) \subseteq \mathscr S'(E)$, $E = \mathbb R^d \simeq T^*\mathbb R^n$, $d = 2n$, obtained in~\cite[Proposition 5.1]{Sj08}. To this end, let $\phi(X,Y)$ be a holomorphic quadratic form on $\mathbb C^{d}_X \times \mathbb C^{d}_Y$ such that
\begeq
\label{eq3.1}
{\rm det}\, \phi''_{XY} \neq 0, \quad {\rm Im}\, \phi''_{YY}>0.
\endeq
Similarly to (\ref{eq2.7}), associated to $\phi$ is the complex linear canonical transformation
\begeq
\label{eq3.2}
\kappa_{\phi}: \mathbb C^{2d} \ni (Y,-\phi'_Y(X,Y)) \mapsto (X,\phi'_X(X,Y)) \in \mathbb C^{2d},
\endeq
and the unitary metaplectic FBI-Bargmann transform
\begeq
\label{eq3.3}
\mathcal T_{\phi}: L^2(\mathbb R^d) \rightarrow H_{\widetilde{\Phi}}(\mathbb C^d),
\endeq
defined as in (\ref{eq2.8}) -- (\ref{eq2.10}). It is established in ~\cite[Proposition 5.1]{Sj08} that $a\in \widetilde{S}(m)$ precisely when
\begeq
\label{eq3.4}
\mathbb C^d \ni X \mapsto \frac{e^{-\widetilde{\Phi}(X)} \mathcal T_{\phi}\, a(X)}{m\left(\kappa_{\phi}^{-1}\left(X,\frac{2}{i}\frac{\partial \widetilde{\Phi}}{\partial X}(X)\right)\right)} \in L^{\infty}(\mathbb C^d),
\endeq
and the $L^{\infty}$--norm of the function in (\ref{eq3.4}) is a norm on $\widetilde{S}(m)$.

\medskip
\noindent
In the following discussion, characterization (\ref{eq3.4}) will only be used in the case when
\begeq
\label{eq3.5}
\phi(X,Y) = i(X-Y)^2, \quad X,Y \in \mathbb C^d,
\endeq
and it follows from (\ref{eq2.10}) and (\ref{eq3.2}) that then
\begeq
\label{eq3.6}
\widetilde{\Phi}(X) = ({\rm Im}\, X)^2, \quad \kappa_{\phi}: \mathbb C^{2d} \ni (Y,H) \mapsto \left(Y - \frac{iH}{2},H\right) \in \mathbb C^{2d}.
\endeq
We conclude, in view of (\ref{eq3.4}), (\ref{eq3.6}) that $a\in \widetilde{S}(m)$ if and only if
\begeq
\label{eq3.7}
\abs{\int_{{\mathbb R}^d} e^{-(X-Y)^2} a(Y)\, dY} \leq C e^{({\rm Im}\, X)^2} m({\rm Re}\, X, -2{\rm Im}\, X),\quad X\in \mathbb C^d,
\endeq
for some $C>0$, or equivalently,
\begeq
\label{eq3.8}
\abs{\mathcal F \left(f_T a\right)(\Xi)} \leq C\, m(T,\Xi),\quad (T,\Xi)\in E \times E^*.
\endeq
Here $f_0$ is the real Gaussian on $\mathbb R^d$ given in (\ref{eq2.13}) and $f_T(Y) = f_0(Y-T)$, $Y,T\in \mathbb R^d$.

\medskip
\noindent
Assuming that $a\in \widetilde{S}(m)$ and passing to the FBI-Bargmann transform side by means of (\ref{eq2.7}), (\ref{eq2.9}), we shall next express (\ref{eq3.8}) in terms of the  symplectic short time Fourier transform of $b = a\circ \kappa^{-1} \in \mathscr S'(\Lambda_{\Phi})$, occurring in (\ref{eq2.15}). Consider, for $T,X \in \Lambda_{\Phi}$, where $\chi_T \in \mathscr S(\Lambda_{\Phi})$ is given in (\ref{eq2.14}),
\begeq
\label{eq3.9}
\mathcal F_{\sigma}(\chi_T b)(X) = \frac{1}{\pi^n} \int_{\Lambda_{\Phi}} e^{2i\sigma(X,Y)} \chi_T(Y) b(Y)\, dY, \quad X \in \Lambda_{\Phi}.
\endeq
Let $\sigma_{\mathbb R}$ be the standard symplectic form on $E = T^*\mathbb R^n$. Making the change of variables $Y = \kappa(Y')$, $Y'\in T^*\mathbb R^n$, with $\kappa$ given in (\ref{eq2.7}), and using the fact that the map
\[
\kappa: (T^*\mathbb R^{n},\sigma_{\mathbb R}) \rightarrow (\Lambda_{\Phi},\sigma|_{\Lambda_{\Phi}})
\]
is real canonical, we get, letting $T' = \kappa^{-1}(T)$, $X' = \kappa^{-1}(X)$,
\begin{multline}
\label{eq3.10}
\mathcal F_{\sigma}(\chi_T b)(X) = \frac{1}{\pi^n} \int_{\mathbb R^d} e^{2i\sigma(X,\kappa(Y'))} f_{T'}(Y') a(Y')\, dY' \\
= \frac{1}{\pi^n} \int_{\mathbb R^d} e^{2i\sigma_{\mathbb R}(X',Y')} f_{T'}(Y') a(Y')\, dY' = \frac{1}{\pi^n} \mathcal F(f_{T'}a)(2J^{-1}X').
\end{multline}
Here we have also used that
\begeq
\label{eq3.11}
\sigma_{\mathbb R}(X',Y') = JX'\cdot Y' = -J^{-1}X'\cdot Y', \quad X',Y' \in T^* \mathbb R^n,
\endeq
where $J$ is given in (\ref{eq1.3}) and where we identify $E$ and $E^*$. We get therefore, using (\ref{eq3.8}) and (\ref{eq3.10}),
\begeq
\label{eq3.12}
\abs{\mathcal F_{\sigma}(\chi_T b)(X)}\leq C \norm{a}_{\widetilde{S}(m)}\, m\left(\kappa^{-1}(T), 2J^{-1}\kappa^{-1}(X)\right), \quad X,T \in \Lambda_{\Phi},
\endeq
\begeq
\label{eq3.13}
\abs{\mathcal F_{\sigma}\left(\chi_{-\frac{Y+T}{2}} b\right)\left(\frac{Y-T}{2}\right)} \leq C \norm{a}_{\widetilde{S}(m)} m\left(-\kappa^{-1}\left(\frac{Y+T}{2}\right), J^{-1} \kappa^{-1}(Y-T)\right).
\endeq
Recalling the linear map $q$ in (\ref{eq1.3.1}), we may rewrite (\ref{eq3.13}) as follows,
\begeq
\label{eq3.14}
\abs{\mathcal F_{\sigma}\left(\chi_{-\frac{Y+T}{2}} b\right)\left(\frac{Y-T}{2}\right)} \leq C \norm{a}_{\widetilde{S}(m)} m\left(q\left(-\kappa^{-1}(Y), -\kappa^{-1}(T)\right)\right), \quad Y,T \in \Lambda_{\Phi}.
\endeq

\medskip
\noindent
Combining Theorem \ref{rank_one} with (\ref{eq3.14}), we shall now complete the proof of Theorem \ref{theo_bounded} by some straightforward estimates. Assuming that the order function $m$ is such that the integral operator
\begeq
\label{eq3.15}
(\mathcal Mg)(Y') = \int_{E} m(q(Y',T')) g(T')\, dT'
\endeq
satisfies for some $1\leq p \leq \infty$,
\begeq
\label{eq3.15.1}
\mathcal M \in \mathcal L(L^p(E), L^p(E)),
\endeq
let us write, using (\ref{eq2.14.1}), (\ref{eq2.15}), and (\ref{eq3.14}),
\begin{multline}
\label{eq3.16}
\abs{(a^w u,v)_{L^2(\mathbb R^n)}} \\
\leq C \norm{a}_{\widetilde{S}(m)} \int\!\!\!\int_{(\Lambda_{\Phi})^2} m\left(q\left(-\kappa^{-1}(Y), -\kappa^{-1}(T)\right)\right)\, \abs{(\Pi_{Y,T} \mathcal Tu, \mathcal Tv)_{H_{\Phi}(\mathbb C^n)}}\, dY\, dT \\
 = C \norm{a}_{\widetilde{S}(m)} \int\!\!\!\int_{(\Lambda_{\Phi})^2} m(q(\kappa^{-1}(Y), \kappa^{-1}(T)))\, F(T) G(Y)\, dY\, dT.
\end{multline}
Here $u,v \in \mathscr S(\mathbb R^n)$ and we have set
\begeq
\label{eq3.17}
F(T) = \abs{(\mathcal  Tu, e^{-i\sigma((x,D_x),T)}V_0)_{H_{\Phi}(\mathbb C^n)}}, \quad G(Y) = \abs{(\mathcal  Tv, e^{-i\sigma((x,D_x),Y)}V_0)_{H_{\Phi}(\mathbb C^n)}}.
\endeq
We get using (\ref{eq2.20.1}), writing $T = (t,(2/i)\partial_t \Phi(t)) \in \Lambda_{\Phi}$,
\begeq
\label{eq3.18}
F(T) \leq C\, \int_{\mathbb C^n} \abs{\mathcal Tu(x)}e^{-\Phi(x)}\, e^{-\abs{x-t}^2/C}\, L(dx) = C \left(\abs{\mathcal T u}e^{-\Phi}*e^{-\abs{\cdot}^2/C}\right)(t),
\endeq
and therefore by Young's inequality we get, using the notation (\ref{app8}), (\ref{app10.1}),
\begeq
\label{eq3.19}
\norm{F}_{L^p(\Lambda_{\Phi})} \leq C \norm{e^{-\abs{\cdot}^2/C}}_{L^1(\mathbb C^n)} \norm{\mathcal Tu}_{H^p_{\Phi}(\mathbb C^n)} \leq
C \norm{\mathcal Tu}_{H^p_{\Phi}(\mathbb C^n)}.
\endeq
It follows from (\ref{eq3.15.1}) and (\ref{eq3.19}) that the function
\begeq
\label{eq3.20}
0\leq H(Y) = \int_{\Lambda_{\Phi}} m(q(\kappa^{-1}(Y), \kappa^{-1}(T))) F(T)\, dT \in L^p(\Lambda_{\Phi})
\endeq
and we have
\begin{multline}
\label{eq3.21}
\norm{H}_{L^p(\Lambda_{\Phi})} \leq \norm{\mathcal M}_{\mathcal L(L^p(E), L^p(E))}\, \norm{F}_{L^p(\Lambda_{\Phi})} \\ \leq C \norm{\mathcal M}_{\mathcal L(L^p(E), L^p(E))}\, \norm{\mathcal T u}_{H^p_{\Phi}(\mathbb C^n)}.
\end{multline}

\medskip
\noindent
We get, using (\ref{eq3.16}), (\ref{eq3.21}), and H\"older's inequality,
\begin{multline}
\label{eq3.22}
\abs{(a^w u,v)_{L^2(\mathbb R^n)}} \leq C\norm{a}_{\widetilde{S}(m)}\, \int_{\Lambda_{\Phi}} H(Y)\, G(Y)\, dY \leq C \norm{a}_{\widetilde{S}(m)} \norm{H}_{L^p(\Lambda_{\Phi})}
\norm{G }_{L^q(\Lambda_{\Phi})} \\
\leq C \norm{a}_{\widetilde{S}(m)} \norm{\mathcal M}_{\mathcal L(L^p(E), L^p(E))}\,\norm{\mathcal T u}_{H^p_{\Phi}(\mathbb C^n)}\, \norm{\mathcal T v}_{H^q_{\Phi}(\mathbb C^n)}.
\end{multline}
Here
\[
\frac{1}{p} + \frac{1}{q} = 1, \quad 1\leq p,q \leq \infty,
\]
and we have also used that
\[
\norm{G}_{L^q(\Lambda_{\Phi})} \leq C \norm{\mathcal Tv}_{H^q_{\Phi}(\mathbb C^n)},
\]
similarly to (\ref{eq3.18}), (\ref{eq3.19}). Letting $U = \mathcal T(a^w u)\in \mathcal T(\mathscr S'(\mathbb R^n))$, we get from (\ref{eq3.22}), for all $V \in \mathcal T(\mathscr S(\mathbb R^n))$, i.e. for all $V\in {\rm Hol}(\mathbb C^n)$ such that $V(x) = \mathcal O_N(1) \langle{x\rangle}^{-N} e^{\Phi(x)}$, for all $N \geq 0$,
\begeq
\label{eq3.23}
\abs{(U,V)_{H_{\Phi}(\mathbb C^n)}} \leq C \norm{a}_{\widetilde{S}(m)} \norm{\mathcal M}_{\mathcal L(L^p(E), L^p(E))}\, \norm{\mathcal T u}_{H^p_{\Phi}(\mathbb C^n)}\, \norm{V}_{H^q_{\Phi}(\mathbb C^n)}.
\endeq

\medskip
\noindent
Let next
\[
\Pi_{\Phi}: L^2(\mathbb C^n, e^{-2\Phi} L(dx)) \rightarrow H_{\Phi}(\mathbb C^n)
\]
be the orthogonal projection, and let us recall from~\cite[Section 1]{Sj96},~\cite[Section 1.3]{HiSj15} that $\Pi_{\Phi}$ is given by
\begeq
\label{eq3.24}
\Pi_{\Phi} g(x) = a_{\Phi} \int e^{2\Psi(x,\overline{y})}\, g(y) e^{-2\Phi(y)}\, L(dy), \quad a_{\Phi} > 0.
\endeq
Here $\Psi$ is the unique holomorphic quadratic form on $\mathbb C^{2n}_{x,y}$ such that $\Psi(x,\overline{x}) = \Phi(x)$, $x\in \mathbb C^n$. We have
\begeq
\label{eq3.25}
\Phi(x) + \Phi(y) - 2{\rm Re}\, \Psi(x,\overline{y}) \asymp \abs{x-y}^2, \quad x,y\in \mathbb C^n,
\endeq
see~\cite[Section 1]{Sj96},~\cite[Section 1.3]{HiSj15}, and it follows from (\ref{eq3.25}) that we can take $V = \Pi_{\Phi} \chi$, $\chi \in C_0(\mathbb C^n)$, in (\ref{eq3.23}). We get therefore for all $\chi \in C_0(\mathbb C^n)$,
\begin{multline}
\label{eq3.26}
\abs{(U, \chi)_{L^2_{\Phi}(\mathbb C^n)}} = \abs{(\Pi_{\Phi} U, \chi)_{L^2_{\Phi}(\mathbb C^n)}}
= \abs{(U,\Pi_{\Phi}\chi)_{H_{\Phi}(\mathbb C^n)}} \\
\leq C \norm{a}_{\widetilde{S}(m)} \norm{\mathcal M}_{\mathcal L(L^p(E), L^p(E))}\, \norm{\mathcal T u}_{H^p_{\Phi}(\mathbb C^n)}\, \norm{\chi e^{-\Phi}}_{L^q(\mathbb C^n)}.
\end{multline}
Here we have also used that the reproducing property
\[
\Pi_{\Phi} U = U
\]
holds for all $U \in \mathcal T(\mathscr S'(\mathbb R^n))$, see~\cite[Section 12.2]{Sj02}, and that
\[
\Pi_{\Phi}: L^q(\mathbb C^n, e^{-q\Phi}L(dx)) \rightarrow L^q(\mathbb C^n, e^{-q\Phi}L(dx)), \quad 1\leq q < \infty,
\]
\[
\Pi_{\Phi}: e^{\Phi} L^{\infty}(\mathbb C^n) \rightarrow e^{\Phi} L^{\infty}(\mathbb C^n),
\]
is bounded, in view of (\ref{eq3.24}), (\ref{eq3.25}). It is therefore clear from (\ref{eq3.26}) that $U e^{-\Phi} \in L^p(\mathbb C^n)$, in other words, that $U \in H^p_{\Phi}(\mathbb C^n)$, with
\[
\norm{U}_{H^p_{\Phi}(\mathbb C^n)} = \norm{\mathcal T(a^w u)}_{H^p_{\Phi}(\mathbb C^n)} \leq C \norm{a}_{\widetilde{S}(m)} \norm{\mathcal M}_{\mathcal L(L^p(E), L^p(E))}\,\norm{\mathcal T u}_{H^p_{\Phi}(\mathbb C^n)},
\]
for all $u\in \mathscr S(\mathbb R^n)$. We get, recalling (\ref{app30}),
\[
\norm{a^w u}_{M^p(\mathbb R^n)} \leq C \norm{a}_{\widetilde{S}(m)} \norm{\mathcal M}_{\mathcal L(L^p(E), L^p(E))}\,\norm{u}_{M^p(\mathbb R^n)},
\]
which completes the proof of Theorem \ref{theo_bounded}.

\section{Effective kernels and composition}
\label{sect_compos}
\medskip
\noindent
We shall start with some general observations and remarks, closely related to the discussion in~\cite[Section 5]{Sj08}. Let $a\in \mathscr S'(T^*\mathbb R^n)$ and let $K_{a^w} \in \mathscr S'(\mathbb R^{n} \times \mathbb R^n)$ be the Schwartz kernel of $a^w$, so that we have for all $u,v\in \mathscr S(\mathbb R^n)$,
\begeq
\label{eq4.1}
(a^w u,v)_{L^2(\mathbb R^n)} = \langle{a^w u, \overline{v}\rangle}_{\mathscr S'(\mathbb R^n), \mathscr S(\mathbb R^n)} = \langle{K_{a^w}, \overline{v}\otimes u\rangle}_{\mathscr S'(\mathbb R^{2n}), \mathscr S(\mathbb R^{2n})} = (K_{a^w}, v\otimes \overline{u})_{L^2(\mathbb R^{2n})}.
\endeq
Let $\mathcal T: L^2(\mathbb R^n) \rightarrow H_{\Phi}(\mathbb C^n)$ be a metaplectic FBI-Bargmann transform as in (\ref{eq2.8}), (\ref{eq2.9}). Then, as observed in~\cite[Section 5]{Sj08}, the map
\begeq
\label{eq4.2}
\widetilde{\mathcal T}u(x) = C \int e^{-i\varphi^*(x,y)} u(y)\, dy = \overline{\left(\mathcal T \overline{u}\right)(\overline{x})}, \quad \varphi^*(x,y) = \overline{\varphi(\overline{x},\overline{y})}
\endeq
is an FBI-Bargmann transform
\begeq
\label{eq4.3}
\widetilde{\mathcal T}: L^2(\mathbb R^n) \rightarrow H_{\Phi^*}(\mathbb C^n), \quad \Phi^*(x) = \Phi(\overline{x}).
\endeq
It follows that the map $\mathcal T \otimes \widetilde{\mathcal T}$ is a unitary FBI-Bargmann transform,
\begeq
\label{eq4.4}
\mathcal T \otimes \widetilde{\mathcal T}: L^2(\mathbb R^{2n}) \rightarrow H_{F}(\mathbb C^{2n}), \quad F(x,z) = \Phi(x) + \Phi^*(z),\quad (x,z)\in \mathbb C^{n}\times \mathbb C^n,
\endeq
and we get therefore in view of (\ref{eq4.1}), letting $K = (\mathcal T \otimes \widetilde{\mathcal T})(K_{a^w})$,
\begin{multline}
\label{eq4.5}
(a^w u,v)_{L^2(\mathbb R^n)} = \left(K, (\mathcal T \otimes \widetilde{\mathcal T})(v\otimes \overline{u})\right)_{H_{F}(\mathbb C^{2n})} \\
= \left(K, \mathcal T v \otimes \widetilde{\mathcal T}(\overline{u})\right)_{H_{F}(\mathbb C^{2n})} = \int\!\!\!\int K(x,z) (\mathcal  T u)(\overline{z})\, \overline{(\mathcal Tv)(x)}\, e^{-2F(x,z)}\, L(dx)\, L(dz) \\
= \int\!\!\!\int K(x,\overline{z}) (\mathcal  T u)(z)\, \overline{(\mathcal Tv)(x)}\, e^{-2\Phi(x)}\, e^{-2\Phi(z)}\, L(dx)\, L(dz).
\end{multline}
Here the kernel $K = K(x,z)$ is holomorphic on $\mathbb C^{2n}_{x,z}$ such that $e^{-\Phi(x)} K(x,\overline{z}) e^{-\Phi(z)}$ is of temperate growth at infinity, and using (\ref{eq4.5}) we get therefore for all $x\in \mathbb C^n$,
\begeq
\label{eq4.6}
\mathcal T(a^w u)(x) = \int K(x,\overline{z}) (\mathcal  T u)(z)\, e^{-2\Phi(z)}\, L(dz),\quad u\in \mathscr S(\mathbb R^n).
\endeq
On the other hand, we get from (\ref{eq2.14.1}), (\ref{eq2.15}),
\begeq
\label{eq4.7}
(a^w u,v)_{L^2(\mathbb R^n)} = \int\!\!\!\int \widetilde{K}(x,\overline{z}) (\mathcal  T u)(z)\, \overline{(\mathcal Tv)(x)}\, e^{-2\Phi(x)}\, e^{-2\Phi(z)}\, L(dx)\, L(dz),
\endeq
where, letting $V_Y = e^{i\sigma((x,D_x),Y)}V_0$, $V_T = e^{i\sigma((x,D_x),T)}V_0$,
\begeq
\label{eq4.8}
\widetilde{K}(x,z)
= \frac{1}{2^{n} (2\pi)^{2n}}\int\!\!\!\int_{(\Lambda_{\Phi})^2} e^{\frac{i\sigma(Y,T)}{2}}
\mathcal F_{\sigma}\left(\chi_{-\frac{Y+T}{2}} b\right)\left(\frac{Y-T}{2}\right) V_Y(x)\, \overline{V_T(\overline{z})}\, dY\, dT.
\endeq
Using (\ref{eq2.20.1}), we observe that the integral in (\ref{eq4.8}) converges absolutely and locally uniformly in $(x,z)$, defining an entire holomorphic function $\widetilde{K}(x,z)$ on $\mathbb C^{2n}_{x,z}$, such that $e^{-\Phi(x)} \widetilde{K}(x,\overline{z}) e^{-\Phi(z)}$ is of temperate growth at infinity. It follows therefore, using also (\ref{eq4.7}), that we have for all $x\in \mathbb C^n$,
\begeq
\label{eq4.8.1}
\mathcal T(a^w u)(x) = \int \widetilde{K}(x,\overline{z}) (\mathcal  T u)(z)\, e^{-2\Phi(z)}\, L(dz),\quad u\in \mathscr S(\mathbb R^n),
\endeq
and we get, in view of (\ref{eq4.6}), (\ref{eq4.8.1}), that $K = \widetilde{K}$, i.e.
\begin{multline}
\label{eq4.8.2}
K(x,z) = (\mathcal T \otimes \widetilde{\mathcal T})(K_{a^w})(x,z) \\
= \frac{1}{2^{n} (2\pi)^{2n}}\int\!\!\!\int_{(\Lambda_{\Phi})^2} e^{\frac{i\sigma(Y,T)}{2}}
\mathcal F_{\sigma}\left(\chi_{-\frac{Y+T}{2}} b\right)\left(\frac{Y-T}{2}\right) V_Y(x)\, \overline{V_T(\overline{z})}\, dY\, dT.
\end{multline}
Following~\cite[Section 5]{Sj08}, we shall refer to the function
\begeq
\label{eq4.9}
K^{\rm eff}_{a^w}(x,z) = e^{-\Phi(x)} K(x,\overline{z}) e^{-\Phi(z)}
\endeq
as the effective kernel of $a^w$.

\medskip
\noindent
Let us write next
\begeq
\label{eq4.9.1}
K_{a^w}(x,y) = \left(\mathcal F^{-1}_2 a\right)\left(\frac{x+y}{2},x-y\right) = \left(\gamma^* \mathcal F^{-1}_2 a\right)(x,y).
\endeq
Here $\mathcal F^{-1}_2 a$ is the inverse Fourier transform of $a$ with respect to the second variable, and
\[
\gamma(x,y) = \left(\frac{x+y}{2},x-y\right), \quad \gamma^{-1}(x,y) = \left(x + \frac{y}{2}, x - \frac{y}{2}\right).
\]
The map
\begeq
\label{eq4.9.2}
U = \gamma^* \circ \mathcal F^{-1}_2: \mathscr S'(T^* \mathbb R^n) \ni a\mapsto K_{a^w} \in \mathscr S'(\mathbb R^n \times \mathbb R^n)
\endeq
can be viewed therefore as a metaplectic Fourier integral operator with real phase, and the associated linear canonical transformation is given by $\kappa_U = \kappa_{\gamma^*} \circ \kappa_{\mathcal F^{-1}_2}$. Let us compute $\kappa_U$, by a small modification of an argument in~\cite[Section 5]{Sj08}. To this end, we have, writing $a = a(t,\tau)$, $E = T^* \mathbb R^n$,
\begeq
\label{eq4.9.3}
\kappa_{\mathcal F^{-1}_2}: E \times E^* \ni (t,\tau;t^*,\tau^*) \mapsto (t,-\tau^*;t^*,\tau) \in T^* \mathbb R^{2n},
\endeq
\[
\kappa_{\gamma^*}: T^* \mathbb R^{2n} \ni (X,X^*) \mapsto (\gamma^{-1}X, \gamma^t X^*) \in T^*\mathbb R^{2n}, \quad X = (x,y) \in \mathbb R^{2n},
\]
or in other words,
\begeq
\label{eq4.9.4}
\kappa_{\gamma^*}: T^* \mathbb R^{2n} \ni (x,y;x^*,y^*) \mapsto \left(x + \frac{y}{2}, x - \frac{y}{2}; \frac{x^*}{2} + y^*, \frac{x^*}{2} - y^*\right)\in T^* \mathbb R^{2n}.
\endeq
We get, using (\ref{eq4.9.3}), (\ref{eq4.9.4}),
\begeq
\label{eq4.9.5}
\kappa_U: (t,\tau;t^*,\tau^*) \mapsto \left(t - \frac{\tau^*}{2}, t + \frac{\tau^*}{2}; \frac{t^*}{2} + \tau, \frac{t^*}{2} - \tau\right) = (x,y;x^*,y^*),
\endeq
and writing $(x,x^*;y,y^*)$ rather than $(x,y;x^*,y^*)$, we obtain the following form for $\kappa_U$,
\begeq
\label{eq4.9.6}
\kappa_U: E \times E^* \ni (t,\tau;t^*,\tau^*) \mapsto \left(t - \frac{\tau^*}{2}, \tau + \frac{t^*}{2}; t + \frac{\tau^*}{2}, -\tau + \frac{t^*}{2}\right) \in T^* \mathbb R^n \times T^* \mathbb R^n.
\endeq

\medskip
\noindent
We next observe that the composition
\begeq
\label{eq4.9.7}
\left(\mathcal T \otimes \widetilde{\mathcal T}\right)\circ U: L^2(E) \rightarrow H_F(\mathbb C^{2n})
\endeq
mapping $a\in \mathscr S'(E)$ to the kernel $K$ in (\ref{eq4.6}), is a metaplectic FBI-Bargmann transform. Indeed, a change of variables shows that $\left(\mathcal T \otimes \widetilde{\mathcal T}\right)\circ \gamma^*$ is such a transform, and in Appendix \ref{modFBI} we verified the general well known fact that the composition of an FBI-Bargmann transform and the (partial) Fourier transformation is also an FBI-Bargmann transform.

\medskip
\noindent
A direct computation using (\ref{eq4.2}) shows that the complex linear canonical transformation associated to $\mathcal T \otimes \widetilde{\mathcal T}$ is given by $\kappa \times (\Gamma \kappa \Gamma)$, where $\kappa$ is defined in (\ref{eq2.7}) and the antilinear involution $\Gamma$ is given in (\ref{app15.1}), see also~\cite[Section 5]{Sj08}. It follows that the complex linear canonical transformation associated to the Bargmann transform in (\ref{eq4.9.7}) is given by $(\kappa \times (\Gamma \kappa \Gamma))\circ \kappa_U$, and we get in view of (\ref{eq4.9.6}), recalling the standard symplectic matrix $J$ in (\ref{eq1.3}),
\begin{multline}
\label{eq4.9.8}
(\kappa \times (\Gamma \kappa \Gamma))\circ \kappa_U: (E \times E^*)^{\mathbb C} \ni (t,\tau; t^*,\tau^*) \\
\mapsto \left(\kappa\left((t,\tau) - \frac{1}{2}J(t^*,\tau^*)\right), \Gamma \kappa \left(\overline{(t,\tau)} + \frac{1}{2} \overline{J(t^*,\tau^*)}\right)\right) \in T^* \mathbb C^{n} \times T^* \mathbb C^{n}.
\end{multline}
In particular,
\[
\left((\kappa \times (\Gamma \kappa \Gamma))\circ \kappa_U\right)(E \times E^*) = \Lambda_{\Phi} \times \Gamma(\Lambda_{\Phi}) = \Lambda_{\Phi} \times \Lambda_{\Phi^*}.
\]

\medskip
\noindent
We summarize the discussion above in the following proposition, which has been established in~\cite[Section 5]{Sj08}, apart from the expression (\ref{eq4.8.2}) for the effective kernel of $a^w$. See also~\cite[Appendix A]{CoHiSj19} for a closely related discussion.
\begin{prop}
\label{eff_kernel}
Let $a\in \mathscr S'(T^* \mathbb R^n)$ and let $\mathcal T: L^2(\mathbb R^n) \rightarrow H_{\Phi}(\mathbb C^n)$ be an FBI-Bargmann transform as in {\rm (\ref{eq2.7})}, {\rm (\ref{eq2.8})}. There exists a unique kernel $K \in {\rm Hol}(\mathbb C^n_x \times \mathbb C^n_z)$, such that $e^{-\Phi(x)} K(x,\overline{z})e^{-\Phi(z)}$ is of temperate growth at infinity, which satisfies for all $x\in \mathbb C^n$,
\[
\mathcal T(a^w u)(x) = \int K(x,\overline{z}) (\mathcal  T u)(z)\, e^{-2\Phi(z)}\, L(dz),\quad u\in \mathscr S(\mathbb R^n).
\]
The kernel $K$ is given by {\rm (\ref{eq4.8.2})}, and the map $\mathscr S'(T^* \mathbb R^n) \ni a \mapsto K\in {\rm Hol}(\mathbb C^{2n})$ is a metaplectic FBI-Bargmann transform of the form {\rm (\ref{eq4.9.7})}, with the associated complex linear canonical transformation given by {\rm (\ref{eq4.9.8})}.
\end{prop}

\medskip
\noindent
{\it Example}. Let $m>0$ be an order function on $E = T^* \mathbb R^n \simeq \mathbb R^{2n}$, and following~\cite[Chapter 7]{DiSj}, let us consider the symbol class
\[
S(m) = \left\{a\in C^{\infty}(E);\, m^{-1} \partial^{\alpha} a \in L^{\infty}(E),\,\,\forall \alpha \in \mathbb N^{2n}\right\}.
\]
We shall estimate the effective kernel of $a^w$, for $a\in S(m)$. To this end, carrying out repeated integration by parts and using the fact that $m$ is an order function, we get for all $N = 1,2,\ldots ,$
\begin{equation}
\label{eq4.9.9}
\abs{\mathcal F\left(f_{T'} a\right)(X')} \leq C_N\, m(T') \langle{X'\rangle}^{-N}, \quad T' \in E, \,\, X' \in E^*.
\end{equation}
Here $f_{T'}(Y') = f_0(Y'-T')$ and $f_0$ is given in (\ref{eq2.13}). In view of (\ref{eq3.10}) and (\ref{eq4.9.9}), we obtain that
\begeq
\label{eq4.9.9.1}
\abs{\mathcal F_{\sigma}\left(\chi_{-\frac{Y+T}{2}} b\right)\left(\frac{Y-T}{2}\right)} \leq C_N m\left(-\kappa^{-1}\left(\frac{Y+T}{2}\right)\right)\langle{Y-T\rangle}^{-N}, \quad Y,T \in \Lambda_{\Phi}.
\endeq
We get therefore, using (\ref{eq4.8.2}), (\ref{eq4.9}), (\ref{eq2.20.1}), and (\ref{eq4.9.9.1}), and the fact that $m$ is an order function,
\begin{multline}
\label{eq4.9.9.2}
\abs{K^{\rm eff}_{a^w}(x,z)} \leq C_N \int\!\!\!\int m\left(\kappa^{-1}\left(\frac{Y+T}{2}\right)\right)\langle{y-t\rangle}^{-N} e^{-\abs{x-y}^2/C} e^{-\abs{z-t}^2/C}\, L(dy)\, L(dt) \\
\leq C_N\, m\left(\kappa^{-1}\left(\frac{X+Z}{2}\right)\right)\langle{x-z\rangle}^{-N}, \quad N =1,2,\ldots
\end{multline}
Here $X, Z \in \Lambda_{\Phi}$ are the points above $x$, $z \in \mathbb C^n$, respectively. See also~\cite[Example 4.4]{Sj08}.

\bigskip
\noindent
Let next $m$ be an order function on $E \times E^*$, where $E = T^* \mathbb R^n$. Assuming that $a\in \widetilde{S}(m)$, we shall estimate the effective kernel of $a^w$. When doing so, we write using (\ref{eq3.14}), (\ref{eq2.20.1}), (\ref{eq4.8.2}), and (\ref{eq4.9}),
\begin{multline}
\label{eq4.10}
\abs{K^{\rm eff}_{a^w}(x,z)} \leq C\, \norm{a}_{\widetilde{S}(m)} \int\!\!\!\int m \left(q\left(\kappa^{-1}(Y), \kappa^{-1}(T)\right)\right) e^{-\abs{x-y}^2/C} e^{-\abs{z-t}^2/C}\, L(dy)\, L(dt) \\
= C\, \norm{a}_{\widetilde{S}(m)} \int\!\!\!\int m(\widetilde{q}(y,t)) e^{-\abs{(x,z)-(y,t)}^2/C}\,L(dy)\, L(dt).
\end{multline}
Here $Y = (y,(2/i)\partial_y\Phi(y))\in \Lambda_{\Phi}$, $T = (t,(2/i)\partial_t\Phi(t))\in \Lambda_{\Phi}$, and $\widetilde{q}$ is the real linear bijection
\[
\widetilde{q}: \mathbb C^n \times \mathbb C^n \rightarrow E \times E^*
\]
given by
\begeq
\label{eq4.10.1}
\widetilde{q}(y,t) = q\left(\kappa^{-1}(Y), \kappa^{-1}(T)\right) = \left(\kappa^{-1}\left(\frac{Y + T}{2}\right), J^{-1} \kappa^{-1}(T-Y)\right).
\endeq
A simple computation using (\ref{eq4.9.8}) shows that
\begeq
\label{eq4.10.2}
\left((\kappa \times (\Gamma \kappa \Gamma))\circ \kappa_U\right)(\widetilde{q}(x,\overline{z})) = \left(\left(x,\frac{2}{i}\frac{\partial \Phi}{\partial x}(x)\right), \left(z,\frac{2}{i}\frac{\partial \Phi^*}{\partial z}(z)\right)\right), \quad (x,z) \in \mathbb C^n \times \mathbb C^n.
\endeq

\medskip
\noindent
We obtain from (\ref{eq4.10}), using also that $m$ is an order function,
\begeq
\label{eq4.11}
\abs{K^{\rm eff}_{a^w}(x,z)} \leq C\, \norm{a}_{\widetilde{S}(m)} m(\widetilde{q}(x,z)), \quad (x,z) \in \mathbb C^n \times \mathbb C^n.
\endeq
Conversely, if $K\in {\rm Hol}(\mathbb C^n \times \mathbb C^n)$ is such that
\[
e^{-\Phi(x)}\abs{K(x,\overline{z})}e^{-\Phi(z)} \leq C\, m(\widetilde{q}(x,z)), \quad (x,z) \in \mathbb C^n \times \mathbb C^n,
\]
for some order function $m$ on $E \times E^*$, then it follows from~\cite[Proposition 5.1]{Sj08} and (\ref{eq4.10.2}) that
\[
K = \left(\mathcal T \otimes \widetilde{\mathcal T}\right)\circ U(a),
\]
for a unique $a\in \widetilde{S}(m)$, and thus, $e^{-\Phi(x)}K(x,\overline{z})e^{-\Phi(z)}$ is the effective kernel of $a^w$.

\medskip
\noindent
Using (\ref{eq4.6}), (\ref{eq4.11}) as our starting point, we shall now consider the action of pseudodifferential operators on $\widetilde{S}(m)$--spaces, where $m$ is an order function on $E = \mathbb R^n \times (\mathbb R^n)^*$. When doing so, we say that a linear map $T: \widetilde{S}(m) \rightarrow \widetilde{S}(\widehat{m})$ is weakly sequentially continuous if for each bounded sequence $u_j \in \widetilde{S}(m)$ such that $u_j \rightarrow 0$ in $\mathscr S'(\mathbb R^n)$, we have $Tu_j$ is bounded in $\widetilde{S}(\widehat{m})$ and $Tu_j \rightarrow 0$ in $\mathscr S'(\mathbb R^n)$.

\medskip
\noindent
The following result is closely related to~\cite[Theorem 7.11]{Sj08}, and the only possibly new point seems to be the uniqueness and the continuity properties of the extension.
\begin{theo}
\label{theo_bound_symbol}
Let $E = \mathbb R^n \times (\mathbb R^n)^*$, and let $m_1$ and $m_2$ be order functions on $E \times E^*$ and on $E$, respectively. Recalling the linear bijection $q$ in {\rm (\ref{eq1.3.1})}, assume that the integral
\begeq
\label{eq4.12}
m_3(x) = \int_E m_1(q(x,y))\, m_2(y)\,dy
\endeq
converges for at least one value of $x \in E$. Then it converges for all values and defines an order function $m_3$ on $E$. Let $a\in \widetilde{S}(m_1)$. Then the operator $a^w: \mathscr S(\mathbb R^n) \rightarrow \mathscr S'(\mathbb R^n)$ extends uniquely to a linear weakly sequentially continuous map
\begeq
\label{eq4.12.1}
a^w: \widetilde{S}(m_2) \rightarrow \widetilde{S}(m_3),
\endeq
which satisfies
\begeq
\label{eq4.12.2}
\norm{a^w u}_{\widetilde{S}(m_3)} \leq \mathcal O(1) \norm{a}_{\widetilde{S}(m_1)} \norm{u}_{\widetilde{S}(m_2)}, \quad u\in \widetilde{S}(m_2).
\endeq
\end{theo}
\begin{proof}
The first statement concerning the integral in (\ref{eq4.12}) is clear, since $m_1$ is an order function. Next, let $u\in \widetilde{S}(m_2)$ and recall from~\cite[Proposition 5.1]{Sj08} that
\begeq
\label{eq4.13}
\abs{\mathcal T u(z)} \leq \mathcal O(1) \norm{u}_{\widetilde{S}(m_2)} e^{\Phi(z)} m_2\left(\kappa^{-1}\left(z,\frac{2}{i}\frac{\partial \Phi}{\partial z}(z)\right)\right), \quad z\in \mathbb C^n.
\endeq
Here $\mathcal T$ is an FBI-Bargmann transform as in (\ref{eq2.8}), (\ref{eq2.9}). It follows from (\ref{eq4.9}), (\ref{eq4.11}), (\ref{eq4.13}), and the assumption in (\ref{eq4.12}) that for each $x\in \mathbb C^n$, we have
\[
\mathbb C^n \ni z \mapsto K(x,\overline{z}) \mathcal Tu(z) e^{-2\Phi(z)} \in L^1(\mathbb C^n),
\]
with the $L^1$--norm bounded by
\begeq
\label{eq4.14}
\mathcal O(1) \norm{a}_{\widetilde{S}(m_1)}\, \norm{u}_{\widetilde{S}(m_2)} e^{\Phi(x)} m_3\left(\kappa^{-1}\left(x,\frac{2}{i}\frac{\partial \Phi}{\partial x}(x)\right)\right).
\endeq
We also see that the function
\[
\mathbb C^n \ni x\mapsto \int K(x,\overline{z}) \mathcal Tu(z) e^{-2\Phi(z)}\, L(dz)
\]
is entire, of modulus not exceeding (\ref{eq4.14}). Another application of~\cite[Proposition 5.1]{Sj08} shows therefore that the tempered distribution $a^w u$ is well defined as an element of $\widetilde{S}(m_3)$, and (\ref{eq4.6}), (\ref{eq4.12.2}) hold. It remains for us to prove therefore that the map $\widetilde{S}(m_2) \ni u\mapsto a^w u \in \widetilde{S}(m_3)$ is weakly sequentially continuous, for then the density result of Proposition \ref{density_prop} will show that the extension is unique. To this end, let $u_j \in \widetilde{S}(m_2)$ be a bounded sequence such that $u_j \rightarrow 0$ in $\mathscr S'(\mathbb R^n)$, and let us write for $\varphi \in \mathscr S(\mathbb R^n)$ using (\ref{eq4.6}),
\begin{multline}
\label{eq4.15}
(a^w u_j, \varphi)_{L^2(\mathbb R^n)} = (\mathcal T(a^w u_j), \mathcal T\varphi)_{H_{\Phi}(\mathbb C^n)} \\
= \int\!\!\!\int K(x,\overline{z}) \mathcal Tu_j(z)\, \overline{\mathcal T\varphi(x)}\, e^{-2\Phi(z)}\, e^{-2\Phi(x)}\, L(dz)\, L(dx).
\end{multline}
Here the integrand tends to $0$ pointwise and the boundedness of $u_j$ in $\widetilde{S}(m_2)$ shows that we have for all $N$, uniformly in $j$,
\begin{multline}
\label{eq4.16}
\abs{K(x,\overline{z}) \mathcal Tu_j(z)\, \mathcal T\varphi(x)}\, e^{-2\Phi(z)}\, e^{-2\Phi(x)} \\
\leq \mathcal O_N(1)\, m_1(\widetilde{q}(x,z))\, m_2\left(\kappa^{-1}\left(z,\frac{2}{i}\frac{\partial \Phi}{\partial z}(z)\right)\right)\, \langle{x\rangle}^{-N} \in L^1(\mathbb C^{2n}_{x,z}),
\end{multline}
for all $N>1$ large enough. An application of dominated convergence gives therefore that $a^w u_j \rightarrow 0$ in $\mathscr S'(\mathbb R^n)$.
\end{proof}

\medskip
\noindent
{\it Remark}. Similar arguments show, under the assumptions of Theorem \ref{theo_bound_symbol}, that if $a_j$ is bounded in $\widetilde{S}(m_1)$
and $a_j \rightarrow 0$ in $\mathscr S'(T^* \mathbb R^n)$ then for each $u\in \widetilde{S}(m_2)$, we have $a^w_j u$ is bounded in $\widetilde{S}(m_3)$ and $a^w_j u \rightarrow 0$ in $\mathscr S'(\mathbb R^n)$.

\bigskip
\noindent
Proceeding in the same spirit, we shall next discuss the composition of pseudodifferential operators, leading to Theorem \ref{theo_compos}. Let $a_1 \in \widetilde{S}(m_1)$, $a_2 \in \widetilde{S}(m_2)$, where $m_1$, $m_2$ are order functions on $E \times E^*$, $E = T^* \mathbb R^n$, such that the integral in (\ref{eq1.6}) converges for at least one value of $(x,y) \in E \times E$. Using that $m_1$, $m_2$ are order functions, it is then easy to see that it converges for all values and defines an order function $m_3$ on $E \times E^*$. See also~\cite[Proposition 4.1]{Sj08}.

\medskip
\noindent
Let us write for $u\in \mathscr S(\mathbb R^n)$, using (\ref{eq4.6}), (\ref{eq4.9}),
\begeq
\label{eq4.17}
\mathcal T(a_2^w u)(x)e^{-\Phi(x)} = \int K^{{\rm eff}}_{a^w_2}(x,z) \mathcal Tu(z) e^{-\Phi(z)}\, L(dz).
\endeq
We get, in view of (\ref{eq4.11}), for all $N \geq 0$,
\begeq
\label{eq4.18}
\abs{\mathcal T(a_2^w u)(x)e^{-\Phi(x)}} \leq C_ N(u) \norm{a_2}_{\widetilde{S}(m_2)} \int m_2(\widetilde{q}(x,z)) \langle{z\rangle}^{-N}\, L(dz),
\endeq
where $C_N$ are continuous seminorms on $\mathscr S(\mathbb R^n)$. Taking $N>1$ large enough fixed, we see that the integral in (\ref{eq4.18}) converges for all $x\in \mathbb C^n$ and defines an order function $f_N$ on $E = \mathbb R^n \times (\mathbb R^n)^*$, given by
\begeq
\label{eq4.18.1}
f_N\left(\kappa^{-1}\left(x,\frac{2}{i}\frac{\partial \Phi}{\partial x}(x)\right)\right) := \int m_2(\widetilde{q}(x,z)) \langle{z\rangle}^{-N}\, L(dz).
\endeq
It follows, in view of (\ref{eq4.18}), (\ref{eq4.18.1}), and ~\cite[Proposition 5.1]{Sj08}, that the operator $a_2^w$ is a well defined linear continuous map
\begeq
\label{eq4.19}
a_2^w: \mathscr S(\mathbb R^n) \rightarrow \widetilde{S}(f_N).
\endeq
Next, using (\ref{eq1.6}) and (\ref{eq4.18.1}), we see that the integral
\begin{multline}
\label{eq4.21}
g_N(x) = \int_E m_1(q(x,z)) f_N(z)\,dz = C \int\!\!\!\int m_1(q(x,z)) m_2(q(z,y))\, \langle{\pi(\kappa(y))\rangle}^{-N}\, dy\, dz \\
= C \int m_3(q(x,y)) \langle{\pi(\kappa(y))\rangle}^{-N}\,dy, \quad C > 0,
\end{multline}
converges for all $x \in E$, for each $N$ large enough fixed, and defines an order function $g_N$ on $E$. Here $\pi: \Lambda_{\Phi}\ni (x,\xi) \mapsto x \in \mathbb C^n_x$ is the projection map. An application of Theorem \ref{theo_bound_symbol} shows that the operator $a^w_1$ has a unique linear weakly sequentially continuous extension
\begeq
\label{eq4.21.1}
a^w_1: \widetilde{S}(f_N) \rightarrow \widetilde{S}(g_N),
\endeq
which is also norm continuous. It follows from (\ref{eq4.19}) and (\ref{eq4.21.1}) that the composition
\begeq
\label{eq4.22}
a^w_1 \circ a^w_2: \mathscr S(\mathbb R^n) \rightarrow \widetilde{S}(g_N)
\endeq
is a well defined linear continuous map, and we get in view of (\ref{eq4.6}) for all $u\in \mathscr S(\mathbb R^n)$,
\begin{multline}
\label{eq4.23}
\mathcal T(a^w_1 a^w_2\, u)(x)e^{-\Phi(x)} = \int K^{{\rm eff}}_{a^w_1}(x,z) \mathcal T(a^w_2 u)(z)\,e^{-\Phi(z)}\, L(dz) \\
= \int\!\!\!\int K^{{\rm eff}}_{a^w_1}(x,z) K^{{\rm eff}}_{a^w_2}(z,y) \mathcal Tu(y)\, e^{-\Phi(y)}\, L(dz)\, L(dy) \\
= \int K^{{\rm eff}}_{a^w_1 a^w_2}(x,y)\mathcal Tu(y)\, e^{-\Phi(y)}\, L(dy).
\end{multline}
Here $K^{{\rm eff}}_{a^w_1 a^w_2}$ is given by the integral
\begeq
\label{eq4.24}
K^{{\rm eff}}_{a^w_1 a^w_2}(x,y) = \int K^{{\rm eff}}_{a^w_1}(x,z) K^{{\rm eff}}_{a^w_2}(z,y)\, L(dz),
\endeq
which converges absolutely for all $(x,y)$, and satisfies, in view of (\ref{eq1.6}) and (\ref{eq4.11}) ,
\begeq
\label{eq4.25}
\abs{K^{{\rm eff}}_{a^w_1 a^w_2}(x,y)} \leq \mathcal O(1) \norm{a_1}_{\widetilde{S}(m_1)}\, \norm{a_2}_{\widetilde{S}(m_2)} m_3(\widetilde{q}(x,y)), \quad (x,y) \in \mathbb C^n \times \mathbb C^n.
\endeq
Furthermore, the function $\mathbb C^{2n} \ni (x,y) \mapsto e^{\Phi(x)}K^{{\rm eff}}_{a^w_1 a^w_2}(x,\overline{y})e^{\Phi(\overline{y})} \in {\rm Hol}(\mathbb C^{2n})$, and we can write therefore, for a unique $a_1 \# a_2 \in \widetilde{S}(m_3)$,
\begeq
\label{eq4.26}
e^{\Phi(x)}K^{{\rm eff}}_{a^w_1 a^w_2}(x,\overline{y})e^{\Phi(\overline{y})} = \left(\left(\mathcal T \otimes \widetilde{\mathcal T}\right)\circ U\right)(a_1 \# a_2)(x,y).
\endeq
It follows from (\ref{eq4.25}) that
\begeq
\label{eq4.27}
\norm{a_1 \# a_2}_{\widetilde{S}(m_3)} \leq \mathcal O(1) \norm{a_1}_{\widetilde{S}(m_1)}\, \norm{a_2}_{\widetilde{S}(m_2)},
\endeq
and we have $a^w_1 \circ a^w_2 = (a_1 \# a_2)^w: \mathscr S(\mathbb R^n) \rightarrow \widetilde{S}(g_N)$, in view of (\ref{eq4.23}). The Weyl composition product
\[
\mathscr S(E) \times \mathscr S(E) \ni (a_1, a_2) \mapsto a_1 \# a_2 \in \mathscr S(E)
\]
has therefore a bilinear extension
\begeq
\label{eq4.28}
\widetilde{S}(m_1) \times \widetilde{S}(m_2) \ni (a_1, a_2) \mapsto a_1 \# a_2 \in \widetilde{S}(m_3),
\endeq
which is norm continuous. We now claim that the bilinear map (\ref{eq4.28}) is weakly sequentially continuous, as defined before the statement of Theorem \ref{theo_compos}. This will show the uniqueness of the extension in (\ref{eq4.28}), in view of Proposition \ref{density_prop}. Thus, let
$a_{1,j}$, $a_{2,j}$ be bounded sequences in $\widetilde{S}(m_1)$, $\widetilde{S}(m_2)$, respectively, and assume that $a_{1,j} \rightarrow a_1\in \widetilde{S}(m_1)$, $a_{2,j} \rightarrow a_2\in \widetilde{S}(m_2)$ in $\mathscr S'(E)$. In view of (\ref{eq4.27}), we only have to show that
$a_{1,j}\#a_{2,j} \rightarrow a_1\# a_2$ in $\mathscr S'(E)$. To this end, let $\varphi \in \mathscr S(E)$ and let us write
\begin{multline}
\label{eq4.29}
(a_{1,j}\#a_{2,j},\varphi)_{L^2(E)} = C (\widehat{\mathcal T}(a_{1,j}\#a_{2,j}),\widehat{\mathcal T}\varphi)_{H_F(\mathbb C^{2n})} \\
= C \int\!\!\!\int K^{{\rm eff}}_{a^w_{1,j} a^w_{2,j}}(x,\overline{y}) \overline{\widehat{\mathcal T}\varphi(x,y)} e^{-\Phi(x) - \Phi(\overline{y})}\, L(dx)\, L(dy), \quad C \neq 0.
\end{multline}
Here $\widehat{\mathcal T} := (\mathcal T \otimes \widetilde{\mathcal T})\circ U$ is the FBI-Bargmann transform introduced in (\ref{eq4.9.7}), the strictly plurisubharmonic quadratic weight $F$ is defined in (\ref{eq4.4}), and we have also used (\ref{eq4.26}). We get, using (\ref{eq4.24}) and (\ref{eq4.29}),
\begin{multline}
\label{eq4.30}
(a_{1,j}\#a_{2,j},\varphi)_{L^2(E)} = \\
C \int\!\!\!\int\!\!\!\int K^{{\rm eff}}_{a^w_{1,j}}(x,z) K^{{\rm eff}}_{a^w_{2,j}}(z,\overline{y}) \overline{\widehat{\mathcal T}\varphi(x,y)} e^{-\Phi(x) - \Phi(\overline{y})}\, L(dx)\, L(dy)\, L(dz).
\end{multline}
Here we have
\[
K^{{\rm eff}}_{a^w_{1,j}}(x,z) K^{{\rm eff}}_{a^w_{2,j}}(z,\overline{y}) \rightarrow K^{{\rm eff}}_{a^w_{1}}(x,z)
K^{{\rm eff}}_{a^w_{2}}(z,\overline{y}),
\]
pointwise, and the absolute value of the integrand in (\ref{eq4.30}) does not exceed, in view of (\ref{eq4.11}),
\begeq
\label{eq4.31}
\mathcal O_N(1) m_1(\widetilde{q}(x,z)) m_2(\widetilde{q}(z,\overline{y})) \langle{(x,y)\rangle}^{-N}, \quad N=1,2,\ldots ,
\endeq
uniformly in $j$. It follows from the assumption in (\ref{eq1.6}) that the function in (\ref{eq4.31}) is in $L^1(\mathbb C^{3n})$, provided that $N>1$ is sufficiently large, and dominated convergence gives that the integral in (\ref{eq4.30}) converges to
\begin{multline*}
\int\!\!\!\int\!\!\!\int K^{{\rm eff}}_{a^w_{1}}(x,z) K^{{\rm eff}}_{a^w_{2}}(z,\overline{y}) \overline{\widehat{\mathcal T}\varphi(x,y)} e^{-\Phi(x) - \Phi(\overline{y})}\, L(dx)\, L(dy)\, L(dz) \\
= \int\!\!\!\int K^{{\rm eff}}_{a^w_{1} a^w_{2}}(x,\overline{y}) \overline{\widehat{\mathcal T}\varphi(x,y)} e^{-\Phi(x) - \Phi(\overline{y})}\, L(dx)\, L(dy).
\end{multline*}
We obtain therefore that $a_{1,j}\#a_{2,j} \rightarrow a_1\# a_2$ in $\mathscr S'(E)$, and the proof of Theorem \ref{theo_compos} is complete.

\begin{appendix}
\section{Modulation spaces and Bargmann transforms}
\label{modFBI}

\medskip
\noindent
The purpose of this appendix is to recall the definition of the modulation spaces by means of metaplectic FBI-Bargmann transforms and to show their independence of the choice of such a transform. See~\cite[Chapter 11]{Gr_book},~\cite{BeOk_book}  for a comprehensive introduction to modulation spaces, and~\cite{SiTo} for their characterization by means of the standard Bargmann transform.

\medskip
\noindent
Let $\varphi$ be a holomorphic quadratic form on $\mathbb C^n_x \times \mathbb C^n_y$ such that
\begeq
\label{app1}
{\rm Im}\, \varphi''_{yy}>0,\quad {\rm det}\, \varphi''_{xy} \neq 0.
\endeq
Associated to $\varphi$ is the complex linear canonical transformation
\begeq
\label{app2}
\kappa_{\varphi}: T^*\mathbb C^{n} \ni (y,-\varphi'_y(x,y)) \mapsto (x,\varphi'_x(x,y))\in T^* \mathbb C^{n},
\endeq
and the metaplectic FBI-Bargmann transform
\begeq
\label{app3}
\mathcal T u(x) = C_{\varphi} \int e^{i\varphi(x,y)}\, u(y)\, dy.
\endeq
The constant $C_{\varphi} > 0$ here is chosen suitably so that the map $\mathcal T$ is unitary,
\begeq
\label{app4}
\mathcal T: L^2(\mathbb R^n) \rightarrow H^2_{\Phi}(\mathbb C^n):={\rm Hol}(\mathbb C^n) \cap L^2(\mathbb C^n, e^{-2\Phi}\,L(dx)),
\endeq
see~\cite[Section 1]{Sj96},~\cite[Section 12.2]{Sj02},~\cite[Section 1.3]{HiSj15}. Here $L(dx)$ is the Lebesgue measure on $\mathbb C^n$, and the strictly plurisubharmonic quadratic form $\Phi$ is given by
\begeq
\label{app5}
\Phi(x) = \sup_{y \in \mathbb R^n} \left(-{\rm Im}\, \,\varphi(x,y)\right).
\endeq
Associated to $\Phi$ is the $2n$-dimensional real linear subspace
\begeq
\label{app5.1}
\Lambda_{\Phi} = \left\{\left(x,\frac{2}{i}\frac{\partial \Phi}{\partial x}(x)\right); \, x\in \mathbb C^n\right\} \subseteq T^* \mathbb C^{n} = \mathbb C^n_x \times \mathbb C^n_{\xi},
\endeq
which is I-Lagrangian and R-symplectic, in the sense that the restriction of the complex symplectic (2,0)--form
\begeq
\label{app6}
\sigma = \sum_{j=1}^n d\xi_j \wedge dx_j
\endeq
on $T^* \mathbb C^{n}$ to $\Lambda_{\Phi}$ is real and non-degenerate. In particular, $\Lambda_{\Phi}$ is maximally totally real and we have
\begeq
\label{app6.1}
\kappa_{\varphi}\left(T^*\mathbb R^{n}\right) = \Lambda_{\Phi}.
\endeq

\medskip
\noindent
We also recall from~\cite[Proposition 6.1]{Horm91} that a holomorphic function $U$ on $\mathbb C^n$ is of the form $U = \mathcal T u$ for some unique $u\in \mathscr S'(\mathbb R^n)$, if and only if we have for some $C>0$ and $N$,
\begeq
\label{app7}
\abs{U(x)} \leq C \langle x\rangle^N e^{\Phi(x)}, \quad x\in \mathbb C^n.
\endeq

\bigskip
\noindent
Let us set, in analogy with (\ref{app4}),
\begeq
\label{app8}
H^p_{\Phi}(\mathbb C^n):={\rm Hol}(\mathbb C^n) \cap L^p(\mathbb C^n, e^{-p\Phi}\,L(dx)), \quad 1 \leq p < \infty,
\endeq
and write for $v\in H^p_{\Phi}(\mathbb C^n)$,
\begeq
\label{app9}
\norm{v}^p_{H^p_{\Phi}(\mathbb C^n)} = \int \abs{v(x)}^p e^{-p\Phi(x)}\, L(dx).
\endeq
An application of the mean value theorem for holomorphic functions combined with H\"older's inequality shows that there exists $C = C_{n,\Phi} > 0$ such that for all $v\in H^p_{\Phi}(\mathbb C^n)$ we have
\begeq
\label{app10}
\abs{v(x)} \leq C \norm{v}_{H^p_{\Phi}(\mathbb C^n)}\langle{x\rangle}^{2n/p} e^{\Phi(x)}, \quad x\in \mathbb C^n.
\endeq
Define also
\begeq
\label{app10.1}
H^{\infty}_{\Phi}(\mathbb C^n):={\rm Hol}(\mathbb C^n) \cap e^{\Phi} L^{\infty}(\mathbb C^n),
\endeq
equipped with the natural norm. It follows from (\ref{app10}) that $H^p_{\Phi}(\mathbb C^n) \subseteq \mathcal T\left(\mathscr S'(\mathbb R^n)\right)$, for all $1\leq p \leq \infty$.

\medskip
\noindent
{\it Remark}. In Section \ref{sect_bounded}, we recalled that the reproducing property
\begeq
\label{app10.2}
\Pi_{\Phi} v = v
\endeq
holds for all $v\in \mathcal T\left(\mathscr S'(\mathbb R^n)\right)$, and thus in particular for all $v\in H^p_{\Phi}(\mathbb C^n)$, $1\leq p \leq \infty$. Here $\Pi_{\Phi}$ is the orthogonal projection introduced in (\ref{eq3.24}). Combining (\ref{eq3.24}), (\ref{eq3.25}), (\ref{app10.2}), and H\"older's inequality, we obtain that (\ref{app10}) can be improved to
\[
\abs{v(x)} \leq C_p \norm{v}_{H^p_{\Phi}(\mathbb C^n)}e^{\Phi(x)}, \quad x\in \mathbb C^n,
\]
for all $v\in H^p_{\Phi}(\mathbb C^n)$, $1\leq p \leq \infty$. It follows that
\[
H^p_{\Phi}(\mathbb C^n) \subseteq H^{p'}_{\Phi}(\mathbb C^n), \quad 1\leq p \leq p' \leq \infty.
\]

\medskip
\noindent
We would like to show that the space
\[
\left \{u\in \mathscr S'(\mathbb R^n); \mathcal T u \in H^p_{\Phi}(\mathbb C^n)\right\}
\]
does not depend on the choice of an FBI-Bargmann transform $\mathcal T$ in (\ref{app3}). To this end, we have the following result.

\begin{prop}
\label{indep_FBI}
Let $\mathcal T$, $\widetilde{\mathcal T}$ be two FBI-Bargmann transforms as in {\rm (\ref{app1})}, {\rm (\ref{app3})}, and let $\Phi$, $\widetilde{\Phi}$ be the associated strictly plurisubharmonic quadratic weights. When $u \in \mathscr S'(\mathbb R^n)$, we have $\mathcal T u \in H^p_{\Phi}(\mathbb C^n)$, for some $1\leq p \leq \infty$, precisely when $\widetilde{\mathcal T} u \in H^p_{\widetilde{\Phi}}(\mathbb C^n)$, with the corresponding equivalence of norms.
\end{prop}
\begin{proof}
Let $u\in \mathscr S(\mathbb R^n)$, and let us write
\begeq
\label{app11}
\widetilde{\mathcal T}u = \widetilde{\mathcal T} \mathcal T^* \mathcal Tu.
\endeq
Here $\mathcal T^*$ is the Hilbert space adjoint of $\mathcal T$ in (\ref{app4}) given by
\begeq
\label{app12}
\mathcal T^*v(y) = C_{\varphi} \int e^{-i\varphi^*(\overline{w},y)} v(w) e^{-2\Phi(w)}\, L(dw), \quad y\in \mathbb R^n,
\endeq
where $\varphi^*(w,y) = \overline{\varphi(\overline{w},\overline{y})}$ is holomorphic. If $\widetilde{\varphi}$ is a holomorphic quadratic form associated to $\widetilde{\mathcal T}$, we get, by an application of the method of exact (quadratic) stationary phase~\cite[Lemma 13.2]{Zw12},
\begeq
\label{app13}
\widetilde{\mathcal T} \mathcal T^* v(x) = C_{\widetilde{\varphi}} C_{\varphi} C \int e^{2q(x,\overline{w})} v(w) e^{-2\Phi(w)}\, L(dw), \quad C \neq 0.
\endeq
Here $v\in \mathcal T(\mathscr S(\mathbb R^n)))$ and $q$ is a holomorphic quadratic form on $\mathbb C^n_x \times \mathbb C^n_z$ given by the critical value
\begeq
\label{app14}
q(x,z) = \frac{i}{2} {\rm vc}_{y} \left(\widetilde{\varphi}(x,y) - \varphi^*(z,y)\right).
\endeq
It is natural to regard the operator in (\ref{app13}) as a metaplectic Fourier integral operator associated to the complex linear canonical transformation $\kappa = \kappa_{\widetilde{\varphi}} \circ \kappa_{\varphi}^{-1}: T^*\mathbb C^{n} \rightarrow T^*\mathbb C^{n}$, satisfying, in view of (\ref{app6.1}),
\begeq
\label{app14.1}
\kappa(\Lambda_{\Phi}) = \Lambda_{\widetilde{\Phi}}.
\endeq
It will also be convenient to introduce the real linear bijection
\begeq
\label{app14.2}
\chi: \mathbb C^n \ni w\mapsto \pi_x\left(\kappa\left(w,\frac{2}{i}\partial_w \Phi(w)\right)\right) \in \mathbb C^n,
\endeq
where $\pi_x: \mathbb C^{2n}\ni (x,\xi) \mapsto x\in \mathbb C^n$ is the natural projection. Following an argument of~\cite[Proposition 3.1]{CoHiSj19}, we shall now show the essentially well known fact that
\begeq
\label{app15}
\widetilde{\Phi}(x) + \Phi(w) - 2\,{\rm Re}\, q(x,\overline{w}) \asymp \abs{x - \chi(w)}^2,\quad (x,w) \in \mathbb C^n_x \times \mathbb C^n_w.
\endeq
See also~\cite[Section 6]{Horm91}. When verifying (\ref{app15}), we set $\Phi^*(z) = \Phi(\overline{z})$ and introduce the antilinear involution
\begeq
\label{app15.1}
\Gamma: \mathbb C^{2n} \ni (z,\zeta) \mapsto (\overline{z},-\overline{\zeta})\in \mathbb C^{2n},
\endeq
mapping $\Lambda_{\Phi}$ bijectively onto $\Lambda_{\Phi^*}$. It follows from (\ref{app14.1}) that
\begeq
\label{app16}
\kappa \circ \Gamma: \Lambda_{\Phi^*}\ni \left(z,\frac{2}{i}\partial_z \Phi^*(z)\right) \rightarrow \left(x,\frac{2}{i}\partial_x \widetilde{\Phi}(x)\right) \in \Lambda_{\widetilde{\Phi}},
\endeq
and continuing to follow~\cite[Proposition 3.1]{CoHiSj19}, let us check that the graph ${\rm Graph}(\kappa \circ \Gamma) \cap \left(\Lambda_{\widetilde{\Phi}} \times \Lambda_{\Phi^*}\right)$ of the map in (\ref{app16}) is a Lagrangian subspace of the real symplectic space $\Lambda_{\widetilde{\Phi}} \times \Lambda_{\Phi^*}$, with respect to the standard symplectic form
\[
\sigma_{x,\xi} + \sigma_{z,\zeta} = \sum_{j=1}^n d\xi_j \wedge dx_j + \sum_{j=1}^n d\zeta_j \wedge dz_j
\]
on $\mathbb C^{2n}_{x,\xi} \times \mathbb C^{2n}_{z,\zeta}$, restricted to $\Lambda_{\widetilde{\Phi}} \times \Lambda_{\Phi^*}$. When doing so, let $(t,s)\in \Lambda_{\Phi^*} \times \Lambda_{\Phi^*}$, and let us compute, writing $\sigma$ for the standard symplectic form on $\mathbb C^{2n}$ given in (\ref{app6}),
\[
\sigma(\kappa(\Gamma(t)), \kappa(\Gamma(s))) + \sigma(t,s) = \sigma(\Gamma(t), \Gamma(s)) + \sigma(t,s) = - \overline{\sigma(t,s)} + \sigma(t,s) = 0,
\]
since $\sigma(t,s)$ is real. Here we have also used that $\sigma (\Gamma t,\Gamma s) = -\overline{\sigma (t,s)}$, in view of (\ref{app15.1}). It follows then from~\cite[Section 2]{MeSj03} that the projection
\[
\pi_{x,z}\left({\rm Graph}(\kappa \circ \Gamma)\cap \left(\Lambda_{\widetilde{\Phi}} \times \Lambda_{\Phi^*}\right)\right)
\]
of the graph ${\rm Graph}(\kappa \circ \Gamma)\cap (\Lambda_{\widetilde{\Phi}} \times \Lambda_{\Phi^*})$ to $\mathbb C^{2n}_{x,z}$, is maximally totally real.

\medskip
\noindent
Next, we get using (\ref{app14}),
\begeq
\label{app17}
\partial_x q(x,z) = \frac{i}{2}\partial_x \widetilde{\varphi}(x,y), \quad \partial_z q(x,z) = - \frac{i}{2}\partial_z \varphi^*(z,y),
\endeq
where $y = y(x,z)\in \mathbb C^n$ is the unique critical point corresponding to the critical value in (\ref{app14}), so that
\begeq
\label{app18}
\partial_y \widetilde{\varphi}(x,y) = \partial_y \varphi^*(z,y).
\endeq
Restricting the attention to points $(x,z) \in \pi_{x,z}\left({\rm Graph}(\kappa \circ \Gamma)\cap \Lambda_{\widetilde{\Phi}}\times \Lambda_{\Phi^*}\right)$, we see that then the critical point $y\in \mathbb R^n$ is real, with
\[
\partial_y \widetilde{\varphi}(x,y) = \partial_y \varphi^*(z,y) = \partial_y \varphi(\overline{z},y) \in \mathbb R^n,
\]
and
\[
\kappa_{\varphi}(y, -\partial_y \varphi(\overline{z},y)) = \left(\overline{z}, \frac{2}{i}(\partial_z \Phi)(\overline{z})\right), \quad
\kappa_{\widetilde{\varphi}}(y,-\partial_y \widetilde{\varphi}(x,y)) = \left(x,\frac{2}{i}\partial_x \widetilde{\Phi}(x)\right).
\]
At the points
\[
(x,z) \in \pi_{x,z}\left({\rm Graph}(\kappa \circ \Gamma)\cap \Lambda_{\widetilde{\Phi}}\times \Lambda_{\Phi^*}\right),
\]
we get therefore, using also (\ref{app17}),
\begeq
\label{app19}
\partial_x q(x,z) = \partial_x \widetilde{\Phi}(x),
\endeq
and
\begeq
\label{app20}
\partial_z q(x,z) = -\frac{i}{2} \partial_z \varphi^*(z,y) = -\frac{i}{2} \overline{\partial_z \varphi(\overline{z},y)} = (\partial_{\overline{z}}\Phi)(\overline{z}) = \partial_z \Phi^*(z).
\endeq
It follows from (\ref{app19}), (\ref{app20}) that the strictly plurisubharmonic quadratic form
\[
F(x,z) = \widetilde{\Phi}(x) + \Phi^*(z) - 2{\rm Re}\, q(x,z), \quad (x,z) \in \mathbb C^n_x \times \mathbb C^n_z,
\]
vanishes to the second order along the maximally totally real subspace
\begeq
\label{app21}
\pi_{x,z}\left({\rm Graph}(\kappa \circ \Gamma)\cap \Lambda_{\widetilde{\Phi}}\times \Lambda_{\Phi^*}\right) \subseteq \mathbb C^n_x \times \mathbb C^n_z,
\endeq
implying that
\begeq
\label{app22}
F(x,z) \asymp {\rm dist}\left((x,z), \pi_{x,z}\left({\rm Graph}(\kappa \circ \Gamma)\cap \Lambda_{\widetilde{\Phi}}\times \Lambda_{\Phi^*}\right)\right)^2.
\endeq
We have therefore shown (\ref{app15}).

\medskip
\noindent
By a density argument, we shall now extend (\ref{app11}) to all $u\in \mathscr S'(\mathbb R^n)$. Indeed, let $u_j \in \mathscr S(\mathbb R^n)$ be such that $u_j \rightarrow u\in \mathscr S'(\mathbb R^n)$ in $\mathscr S'(\mathbb R^n)$, and let us write in view of (\ref{app11}), (\ref{app13}),
\begeq
\label{app23}
\widetilde{\mathcal T}u_j(x) = C \int e^{2q(x,\overline{w})} \mathcal Tu_j(w) e^{-2\Phi(w)}\, L(dw), \quad C \neq 0,\quad j = 1,2\ldots
\endeq
Here $\widetilde{\mathcal T}u_j \rightarrow \widetilde{\mathcal T}u$, $\mathcal T u_j \rightarrow \mathcal T u$ pointwise, and the Banach-Steinhaus theorem gives a uniform bound for some $M \geq 0$,
\begeq
\label{eq24}
\abs{\mathcal Tu_j(w)} \leq C \langle{w\rangle}^M e^{\Phi(w)}, \quad w\in \mathbb C^n,\,\,\, j=1,2,\ldots
\endeq
Recalling also (\ref{app15}), we obtain therefore by dominated convergence,
\begeq
\label{app25}
\widetilde{\mathcal T}u(x) = C \int e^{2q(x,\overline{w})} \mathcal Tu(w) e^{-2\Phi(w)}\, L(dw), \quad u\in \mathscr S'(\mathbb R^n),
\endeq
and hence
\begeq
\label{app26}
\abs{\widetilde{\mathcal T} u(x)}e^{-\widetilde{\Phi}(x)} \leq C \int e^{-\abs{x - \chi(w)}^2/C}\, \abs{\mathcal Tu(w)}\, e^{-\Phi(w)}\, L(dw), \quad u \in \mathscr S'(\mathbb R^n).
\endeq
Assuming that $u\in \mathscr S'(\mathbb R^n)$ is such that $\mathcal Tu\in H_{\Phi}^p(\mathbb C^n)$, for some $1\leq p \leq \infty$, and applying Schur's lemma to (\ref{app26}), in view of the bijectivity of $\chi$ in (\ref{app14.2}), we get that
$\widetilde{\mathcal T}u \in H^p_{\widetilde{\Phi}}(\mathbb C^n)$ and
\begeq
\label{app27}
\norm{\widetilde{\mathcal T}u}_{H^p_{\widetilde{\Phi}}(\mathbb C^n)} \leq C \norm{\mathcal T u}_{H^p_{\Phi}(\mathbb C^n)}.
\endeq
Interchanging the roles of $\mathcal T$, $\widetilde{\mathcal T}$, we complete the proof.
\end{proof}

\medskip
\noindent
It follows from Proposition \ref{indep_FBI} that when $\mathcal T$ is a metaplectic FBI-Bargmann transform as in (\ref{app1}), (\ref{app3}), the modulation space
\begeq
\label{app30}
M^p(\mathbb R^n) = \{u\in \mathscr S'(\mathbb R^n); \mathcal Tu \in H^p_{\Phi}(\mathbb C^n)\}, \quad 1\leq p \leq \infty,
\endeq
is defined unambiguously. When equipped with the norm
\[
\norm{u}_{M^p(\mathbb R^n)} = \norm{\mathcal T u}_{H^p_{\Phi}(\mathbb C^n)},
\]
it becomes a Banach space, and $\mathscr S(\mathbb R^n)$ is a dense subspace, for $p < \infty$. Indeed, if $u\in M^p(\mathbb R^n)$ and $\chi \in C^{\infty}_0(\mathbb C^n)$ is such that $\chi = 1$ in a neighborhood of $0$, we see that the functions
\[
u_j = \mathcal T^*\left(\chi\left(\frac{\cdot}{j}\right)\mathcal T u\right) \in \mathscr S(\mathbb R^n)
\]
satisfy $u_j \rightarrow u$ in $M^p(\mathbb R^n)$. The map
\[
\mathcal T: M^p(\mathbb R^n) \rightarrow H_{\Phi}^p(\mathbb C^n), \quad 1\leq p \leq \infty,
\]
is an isometric bijection.

\medskip
\noindent
{\it Remark}. Let
\[
\mathcal F_0 u(\xi) = \widehat{u}(\xi) = \frac{1}{(2\pi)^{n/2}}\int e^{-ix\cdot \xi}\, u(x)\, dx
\]
be the unitary Fourier transformation on $\mathbb R^n$ and let $\mathcal T$ be an FBI-Bargmann transform, as in (\ref{app1}), (\ref{app3}). It was remarked in~\cite[Appendix A.1]{HelSj} that $\mathcal T \circ \mathcal F_0$ is also an FBI-Bargmann transform. Indeed, if $\varphi$ is a holomorphic quadratic form on $\mathbb C^{2n}$ defining $\mathcal T$, then quadratic stationary phase shows that $\mathcal T \circ \mathcal F_0$ is defined by
the critical value
\begeq
\label{app31}
\psi(x,y) = {\rm vc}_{\eta} \left(\varphi(x,\eta) - y\cdot \eta\right).
\endeq
We have
\[
\psi'_x(x,y) = \varphi'_x(x, \eta(x,y)), \quad \psi''_{xy} = \varphi''_{x\eta}\,\eta'_y(x,y),
\]
where $\eta = \eta(x,y) \in \mathbb C^n$ is the unique critical point corresponding to the critical value in (\ref{app31}). Thus, $\varphi'_{\eta}(x,\eta) = y$, $\varphi''_{\eta \eta}\,\eta'_y = 1$, and recalling (\ref{app1}) we conclude that ${\rm det}\, \psi''_{xy} \neq 0$. We also see that the complex linear canonical transformation
\[
\kappa_{\psi}: \mathbb C^{2n} \ni (y, -\psi'_y(x,y)) \mapsto (x,\psi'_x(x,y)) \in \mathbb C^{2n}
\]
is of the form $\kappa_{\psi} = \kappa_{\varphi} \circ \kappa_{{\mathcal F}_0}$, where $\kappa_{\varphi}$ is given in (\ref{app2}), and
$\kappa_{{\mathcal F}_0}(y,\eta) = (\eta,-y)$. It follows therefore from~\cite[Section 1]{SjHokk},~\cite[Proposition 3.3]{HiZw25} that ${\rm Im}\, \psi''_{yy} > 0$. In this discussion, we could have replaced the Fourier transformation by a more general metaplectic operator with real phase,  see~\cite[Theorem 18.5.9]{H3} and Section \ref{sect_compos}.

\medskip
\noindent
In the special case when
\begeq
\label{app32}
\varphi(x,y) = \frac{i}{2}(x-y)^2 - \frac{i}{4}x^2,
\endeq
a simple computation using (\ref{app31}) shows that $\psi(x,y) = \varphi(-ix,y)$, and we get
\[
(\mathcal T \circ \mathcal F_0 u)(x) = C_{\varphi} \int e^{i\psi(x,y)} u(y)\, dy = \left(\mathcal R \circ \mathcal T u\right)(x),
\]
where $(\mathcal R v)(x) = v(-ix)$. In other words,
\begeq
\label{app33}
\mathcal T \circ \mathcal F_0 = \mathcal R \circ \mathcal T, \quad \mathcal T \circ \mathcal F_0^{-1} = \mathcal R^{-1} \circ \mathcal T.
\endeq
See also~\cite[Appendix A.1]{HelSj},~\cite[Lemma 2.3]{Wick}. The quadratic weight associated to (\ref{app32}) is the radial weight given by
\begeq
\label{app34}
\Phi(x) = {\rm sup}_{y\in \mathbb R^n} \left(-{\rm Im}\, \varphi(x,y)\right) = \frac{\abs{x}^2}{4},
\endeq
and using (\ref{app30}), (\ref{app33}), and (\ref{app34}), we recover the well known fact that
\[
\mathcal F_0: M^p(\mathbb R^n) \rightarrow M^p(\mathbb R^n)
\]
is linear homeomorphism. See~\cite[Theorem 11.3.5]{Gr_book}.

\medskip
\noindent
{\it Remark}. Let $\ell$ be a complex linear form on $T^*\mathbb C^{n}$ whose restriction to $T^*\mathbb R^{n}$ is real, and recalling the complex linear canonical transformation $\kappa_{\varphi}$ in (\ref{app2}), let us set $k = \ell \circ \kappa_{\varphi}^{-1}$. The restriction of the complex linear form $k$ on $T^*\mathbb C^{n}$ to $\Lambda_{\Phi}$ in (\ref{app5.1}) is real, and we have the exact Egorov property
\begeq
\label{app35}
\mathcal T \circ e^{-i\ell(x,D_x)} = e^{-ik(x,D_x)}\circ \mathcal T,
\endeq
see~\cite[Proposition 1.4]{Sj95},~\cite[Theorem 1.4.2]{HiSj15}. Here the magnetic translation operators $e^{-i\ell(x,D_x)}$, $e^{-ik(x,D_x)}$, given explicitly by (\ref{eq2.12.2}), are unitary on $L^2(\mathbb R^n)$, $H_{\Phi}^2(\mathbb C^n)$, respectively. Given $v\in {\rm Hol}(\mathbb C^n)$, we have seen in~\cite[(2.7)]{HiLaSjZe} that
\begeq
\label{app36}
\abs{\left(e^{-\Phi}e^{-ik(x,D_x)}u\right)(x)} = \abs{\left(e^{-\Phi}u\right)(x- k'_{\xi})},\quad x\in \mathbb C^n,
\endeq
see also (\ref{eq2.20}). It follows that the operator
\[
e^{-ik(x,D_x)}: H^p_{\Phi}(\mathbb C^n) \rightarrow H^p_{\Phi}(\mathbb C^n), \quad 1 \leq p \leq \infty,
\]
is an isometric bijection, and in view of (\ref{app35}), the same holds for the map
\[
e^{-i\ell(x,D_x)}: M^p(\mathbb R^n) \rightarrow M^p(\mathbb R^n),\quad 1 \leq p \leq \infty.
\]
See also~\cite[Theorem 11.3.5]{Gr_book}.
\end{appendix}


\begin{thebibliography}{0}

\bibitem[AnLe14]{AnLe} N. Anantharaman and M. L\'eautaud, {\it Sharp polynomial decay rates for the damped wave equation on the torus (With an appendix by St\'ephane Nonnenmacher)}, Anal. PDE {\bf 7} (2014), 159–-214.

\bibitem[B\'eOk29]{BeOk_book} \'A. B\'enyi and K. Okoudjou, {\it Modulation spaces -- with applications to pseudodifferential operators and nonlinear Schr\"odinger equations}, Birkh\"auser/Springer, New York, 2020.

\bibitem[Bo94]{Bo94} J. M. Bony, {\it Op\'erateurs int\'egraux de Fourier et calcul de Weyl-H\"ormander (cas d'une métrique symplectique)}, Journ\'ees "\'Equations aux D\'eriv\'ees Partielles'' (Saint-Jean-de-Monts, 1994), Exp. No. IX, 14 pp., \'Ecole Polytech., Palaiseau, 1994.

\bibitem[Bou97]{Bou97} A. Boulkhemair, {\it Remarks on a Wiener type pseudodifferential algebra and Fourier integral operators},
Math. Res. Lett. {\bf 4} (1997), 53–-67.

\bibitem[Cao${}^*$20]{Wick} G. Cao, J. Li, M. Shen, B. Wick, and L. Yan, {\it A boundedness criterion for singular integral
operators of convolution type on the Fock space}, Adv. Math. {\bf 363} (2020), 33 pp.

\bibitem[CoHiSj19]{CoHiSj19} L. Coburn, M. Hitrik, and J. Sj\"ostrand, {\it Positivity, complex FIOs, and Toeplitz operators},
Pure Appl. Anal. {\bf 1} (2019), 327–-357.

\bibitem[CoGrNiRo13]{Co13} E. Cordero, K. Gr\"ochenig, F. Nicola, and L. Rodino, {\it Wiener algebras of Fourier integral operators}, J. Math. Pures Appl. {\bf 99} (2013), 219–-233.

\bibitem[CoNi10]{CoNi} E. Cordero and F. Nicola, {\it Pseudodifferential operators on $L^p$, Wiener amalgam and modulation spaces},
Int. Math. Res. Not. IMRN {\bf 10} (2010), 1860–-1893.

\bibitem[DiSj99]{DiSj} M. Dimassi and J. Sj\"ostrand, \emph{Spectral asymptotics in the semi-classical limit}, Cambridge University Press, 1999.

\bibitem[Fe83]{Fe83} H. G. Feichtinger, {\it Modulation spaces over locally compact Abelian groups}, Technical Report, University Vienna, 1983.

\bibitem[Gr01]{Gr_book} K. Gr\"ochenig, {\it Foundations of time frequency analysis}, Birkh\"auser, 2001.

\bibitem[Gr06]{Gr06} K. Gr\"ochenig, {\it Time-frequency analysis of Sj\"ostrand's class}, Rev. Mat. Iberoam. {\bf 22} (2006), 703–-724.

\bibitem[GrHe99]{GrHe99} K. Gr\"ochenig and C. Heil, {\it Modulation spaces and pseudodifferential operators}, Integr. Equ. Oper. Theory {\bf 34} (1999), 439--457.

\bibitem[GrRz08]{GrRz08} K. Gr\"ochenig and Z. Rzeszotnik, {\it Banach algebras of pseudodifferential operators and their almost diagonalization},
Ann. Inst. Fourier {\bf 58} (2008), 2279–-2314.

\bibitem[GrTo11]{GrTo11} K. Gr\"ochenig and J. Toft, {\it Isomorphism properties of Toeplitz operators and pseudo-differential operators between modulation spaces}, J. Anal. Math. {\bf 114} (2011), 255–-283.

\bibitem[HelSj89]{HelSj} B.~Helffer and J.~Sj\"ostrand, {\it Semiclassical analysis for Harper's equation. {\rm III}. Cantor structure of the spectrum}, Mem. Soc. Math. France (N.S.) {\bf 39} (1989), 1--124.

\bibitem[HiLaSjZe22]{HiLaSjZe} M. Hitrik, R. Lascar, J. Sj\"ostrand, and M. Zerzeri, {\it Semiclassical Gevrey operators and magnetic translations}, J. Spectr. Theory {\bf 12} (2022), 53–-82.

\bibitem[HiSj15]{HiSj15} M.~Hitrik and J.~Sj\"ostrand, \emph{Two minicourses on analytic microlocal analysis},
"Algebraic and Analytic Microlocal Analysis", Springer Proceedings in Mathematics and Statistics {\bf 269} (2018), 483--540.

\bibitem[HiZw25]{HiZw25} M. Hitrik and M. Zworski, {\it Classically forbidden regions in the chiral model of twisted bilayer graphene (with an appendix by Zhongkai Tao and Maciej Zworski)}, Probab. Math. Phys. {\bf 6} (2025), 505--546.

\bibitem[HoToWa07]{HoToWa07} A. Holst, J. Toft, and P. Wahlberg, {\it Weyl product algebras and modulation spaces},
J. Funct. Anal. {\bf 251} (2007), 463–-491.

\bibitem[Hö85]{H3} L. H{\"o}rmander, \emph{The Analysis of Linear Partial Differential Operators {\rm III}. Pseudo-Differential Operators,\/}
	Springer Verlag, 1985.
	
\bibitem[Hö91]{Horm91} L. H\"ormander, {\it Quadratic hyperbolic operators}, Lecture Notes in Math. {\bf 1495}, 118--160, Springer, Berlin, 1991.

\bibitem[Kl22]{Kl} P. Kleinhenz, {\it Decay rates for the damped wave equation with finite regularity damping},
Math. Res. Lett. {\bf 29} (2022), 1087–-1140.

\bibitem[La01]{La01} D. Labate, {\it Pseudodifferential operators on modulation spaces}, J. Math. Anal. Appl. {\bf 262} (2001), 242–-255.

\bibitem[MeSj03]{MeSj03} A. Melin and J. Sj\"ostrand, {\it Bohr-Sommerfeld quantization condition for non-selfadjoint operators in dimension {\rm 2}}, Ast\'erisque, {\bf 284} (2003), 181–-244.

\bibitem[SiTo12]{SiTo} M. Signahl and J. Toft, {\it Mapping properties for the Bargmann transform on modulation spaces},
J. Pseudo-Differ. Oper. Appl. {\bf 3} (2012), 1–-30.

\bibitem[Sj83]{SjHokk} J. Sj\"ostrand, {\it Analytic wavefront sets and operators with multiple characteristics}, Hokkaido Math. Journal, {\bf 12} (1983), 392--433.

\bibitem[Sj94]{Sj94} J. Sj\"ostrand, {\it An algebra of pseudodifferential operators}, Math. Res. Lett. {\bf 1} (1994), 185–-192.

\bibitem[Sj95]{Sj95} J. Sj\"ostrand, {\it Wiener type algebras of pseudodifferential operators}, S\'eminaire sur les \'Equations aux D\'eriv\'ees Partielles, 1994–1995, Exp. No. IV, 21 pp., \'Ecole Polytech., Palaiseau, 1995.

\bibitem[Sj96]{Sj96} J. Sj\"ostrand, {\it Function spaces associated to global I-Lagrangian manifolds}, Structure of solutions of differential
equations, Katata/Kyoto, 1995, World Scientific, 1996.

\bibitem[Sj02]{Sj02} J. Sj\"ostrand, {\it Lectures on Resonances}, \url{https://sjostrand.perso.math.cnrs.fr/}, 2002.

\bibitem[Sj08]{Sj08} J. Sj\"ostrand, {\it Pseudodifferential operators and weighted normed symbol spaces}, Serdica Math. J. {\bf 34} (2008), 1–-38.

\bibitem[To01]{To01} J. Toft, {\it Subalgebras to a Wiener type algebra of pseudo-differential operators}, Ann. Inst. Fourier {\bf 51} (2001), 1347–-1383.

\bibitem[To04]{To04} J. Toft, {\it Continuity properties for modulation spaces, with applications to pseudo-differential calculus. {\rm I}.},
J. Funct. Anal. {\bf 207} (2004), 399–-429.

\bibitem[To07]{To07} J. Toft, {\it Continuity and Schatten properties for pseudo-differential operators on modulation spaces}, Oper. Theory Adv. Appl. {\bf 172}, 173–206, Birkh\"auser, Basel, 2007.

\bibitem[Zw12]{Zw12} M. Zworski, \emph{Semiclassical Analysis}, AMS, 2012.


\end{thebibliography}
\end{document}